\documentclass[10pt,reqno]{amsart}
\usepackage{url}
\usepackage{amsthm}
\usepackage{fullpage}
\usepackage{amsmath}
\usepackage{amssymb}
\usepackage{amsfonts}

\usepackage{enumitem}
\usepackage{mathscinet}
\usepackage{lipsum}
\usepackage[english]{babel}
\usepackage[autostyle]{csquotes}
\usepackage{bbold}

\usepackage{graphicx}

\usepackage[utf8]{inputenc}
\usepackage[english]{babel}
 
\usepackage{color}
\usepackage{comment}

\newtheorem{thm}{Theorem}[section]
\newtheorem{definition}[thm]{Definition}
\newtheorem{theorem}[thm]{Theorem}
\newtheorem*{oseledec*}{Oseledec Theorem}
\newtheorem{cor}[thm]{Corollary}
\newtheorem{claim}[thm]{Claim}
\newtheorem{lemma}[thm]{Lemma}
\newtheorem{prop}[thm]{Proposition}

\makeatletter
\def\moverlay{\mathpalette\mov@rlay}
\def\mov@rlay#1#2{\leavevmode\vtop{%
   \baselineskip\z@skip \lineskiplimit-\maxdimen
   \ialign{\hfil$\m@th#1##$\hfil\cr#2\crcr}}}
\newcommand{\charfusion}[3][\mathord]{
    #1{\ifx#1\mathop\vphantom{#2}\fi
        \mathpalette\mov@rlay{#2\cr#3}
      }
    \ifx#1\mathop\expandafter\displaylimits\fi}
\makeatother

\newcommand{\bigcupdot}{\charfusion[\mathop]{\bigcup}{\cdot}}
\let\ul\underline

\def\CAR{\ensuremath{\curvearrowright}}
\def\Sig{\ensuremath{\widehat{\Sigma}}}

\def\Sista{\ensuremath{\overset{\star}{\Sigma}}}

\newcommand\mathmiddlescript[1]{\vcenter{\hbox{$\scriptstyle #1$}}}
\def\Sistasharpepsilon{\ensuremath{\overset{\star}{\Sigma}\mathmiddlescript{^{_\epsilon^{^\#}}}}}

\def\Vista{\ensuremath{\overset{\star}{\mathcal{V}}}}
\def\epsista{\ensuremath{\overset{\star}{\epsilon}}}
\def\Eista{\ensuremath{\overset{\star}{\mathcal{E}}}}
\def\Gista{\ensuremath{\overset{\star}{\mathcal{G}}}}
\def\Aista{\ensuremath{\overset{\star}{\mathcal{A}}}}
\def\Zista{\ensuremath{\overset{\star}{Z}}}

\def\HWT{\ensuremath{\mathrm{HWT}}}

\title{Canonically Codable Points and Irreducible Codings}
\author{Snir Ben Ovadia}

\begin{document}
\maketitle
\begin{abstract}
Let $f:M\rightarrow M$ be a $C^{1+\beta}$ diffeomorphism, where $\beta>0$ and $M$ is a compact Riemannian manifold without boundary. In \cite{Sarig}, for all $\chi>0$, for every small enough $\epsilon>0$, Sarig had first constructed a coding $\widehat{\pi}:\Sig\rightarrow M$ which covers the set of all Lyapunov regular $\chi$-hyperbolic points when $\mathrm{dim}M=2$, where $\Sig$ is a topological Markov shift (TMS) over a locally-finite and countable directed graph. $\widehat{\pi}$ is H\"older continuous, and is finite-to-one on $\Sig^\#:=\{\underline{u}\in\Sig:\exists v,w\text{ s.t. }\#\{i\geq0:u_i=v\}=\infty, \#\{i\leq0:u_i=w\}=\infty\}$; and $\widehat{\pi}[\Sig^\#]\supseteq \{\text{Lyapunov regualr and temperable }\chi\text{-hyperbolic points}\}$. We later extended Sarig's result for the case $\mathrm{dim}M\geq2$ in \cite{SBO}. In this work, we offer an improved construction for \cite{SBO} such that ($\forall\epsilon>0$ small enough) we could identify canonically the set $\widehat{\pi}[\Sig^\#]$. We introduce the notions of $\chi$-summable, and $\epsilon$-weakly temperable points.

In \cite{BCS}, the authors show that for each homoclinic class of a periodic hyperbolic point $p$, there exists a maximal irreducible component $\widetilde{\Sigma}\subseteq\Sig$ s.t. all invariant ergodic probability $\chi$-hyperbolic measures which are carried by the homoclinic class of $p$ can be lifted to $\widetilde{\Sigma}$. We use their construction in the context of ergodic homoclinic classes, to show the stronger claim, $\widehat{\pi}[\widetilde{\Sigma}\cap\Sig^\#]=H(p)$ modulo all conservative (possibly infinite) measures ($\mathrm{dim}M\geq2$); where $H(p)$ is the ergodic homoclinic class of $p$, as defined in \cite{RodriguezHertz}, with the (canonically identified) recurrently-codable points replacing the Lyapunov regular points in the definition in \cite{RodriguezHertz}. 
\end{abstract}

%Set some $v_{-\infty}^0$\footnote{If $Leb_{v_{-\infty}^0}(\pi[\Sigma])=Leb_{v_{-\infty}^0}(v_{-\infty}^0)$, then for Leb-a.e $z\in v_{-\infty}^0$ $f(z)=\pi((u_{i-1})_{i\in\mathbb{Z}}), u_i=v_i \forall i\leq0$. Then $f[v_{-\infty}^0]=\cup_{v_0\rightarrow u}V^u((v_{-\infty}^0u|))$ (a finite union).}
%$$\Sigma^-:=\{(u_i)_{i\leq0}: \exists n, \sigma^n((u_i)_{i\leq0})=v_{-\infty}^0\},\sigma((u_i)_{i\leq0}):=(u_{i+1})_{i\leq0}.$$
\tableofcontents
\section{Motivation and Introduction}\label{motivation}
Symbolic dynamics are a powerful tool which allows us to derive many strong conclusions on smooth dynamical systems which admit them. For example, construction of Gibbs and SRB measures for uniformly hyperbolic systems \cite{B4,BowenMarcus,BowenRuelle}, classification of toral automorphisms (in the monumental work of Adler and Weiss) \cite{AW70}, counting periodic points \cite{Sarig}, counting of measures of maximal entropy \cite{BCS}, and many others. In the general setup of a non-uniformly hyperbolic diffeomorphism of a compact boundaryless Riemannian manifold, in \cite{Sarig}, Sarig had constructed a Markov partition when the dimension of the manifold is 2, and we later extended his result to any dimension in \cite{SBO}. Generally, diffeomorphisms might have many points which are not hyperbolic (elliptic islands, homoclinic tangency, etc.); thus the codings of \cite{SBO,Sarig} usually code a set $\subsetneq M$. However, not every point which demonstrates hyperbolic behavior is necessarily covered by these Markov partitions. The authors in both \cite{Sarig,SBO} construct a set of Lyapunov regular points which are covered by the Markov partition; but this comes with two clear disadvantages: 1) Lyapunov regularity restricts us to the study of probability measures, which are known to be carried by the Lyapunov regular points, but do not include the rich ergodic theory of infinite conservative invariant measures (see Definition \ref{conservativeMeasures} for the definition of a conservative measure in our context). 2) There could be many other points which demonstrate the relatively-easy-to-study behavior of the Markov structure, which are not Lyapunov regular, and which are being ignored- thus not using the strength of the coding to its fullest. One could simply choose to work with the image of the coding, or the set covered by the Markov partition, but this prevents us from making any new definitions and extending objects, since defining them based on a specific construction seems less natural. We therefore wish to find a Markov partition (or coding) and a set of hyperbolic points with the following three properties: 1) The Markov partition covers that set of hyperbolic points; 2) Every point which is covered by the Markov partition belongs to that set of hyperbolic points; 3) That set of hyperbolic points is defined canonically, and not by the specific coding (i.e. based on the quality of hyperbolicity of a point along its orbit, and not on a specific choice or construction). We offer such a Markov partition and such a set of hyperbolic points, by presenting an improved way to carry-out the construction in \cite{SBO}, and its analysis. 
This paper treats diffeomorphisms. The case of flows brings in new difficulties, because of issues related to singularities in the Poincar\'e section. The paper \cite{SL14} codes a smaller set than the set of Lyapunov regular points. The author has been informed by Y. Lima that together with J. Buzzi and S. Crovisier he now has a coding which captures a larger set than the set of Lyapunov regular orbits (work in progress). It would be interesting to know if the set of coded points can be characterized completely as we do in this paper for discrete time systems.

In \cite{Smale}, Smale introduced the notion of the homoclinic relation between two orbits of periodic hyperbolic points by, $O\sim O'\iff $ the global stable leaf of a point in $O$ intersects transversely with full codimension the global unstable leaf of a point in $O'$, and vice-versa. In \cite{NewhousePeriodicEquivalenceRelation}, Newhouse showed that this relation is in fact an equivalence relation, and so the notion of a homoclinic class of a periodic hyperbolic point rose naturally. The closure of $\{O': O\sim O'\}$ is a closed and transitive set, and as such it is used often in the studying of transitive hyperbolic dynamics.
In \cite{RodriguezHertz}, Rodriguez-Hertz,
Rodriguez-Hertz, Tahzibi and Ures have introduced a new notion to consider hyperbolic points which are associated with the orbit of a hyperbolic periodic point, which is called an ergodic homoclinic class. They have shown that ergodic homoclinic classes hold the property of admitting at most one SRB measure. We consider this object with the larger set of hyperbolic points (as mentioned in the paragraph above) replacing the Lyapunov regular points.\footnote{This does not harm the uniqueness of SRB measures, as every invariant probability measure is carried by the Lyapunov regular points regardless.} This allows us to not restrict ourselves to probability measures, while having a canonical way of studying conservative (possibly infinite) invariant measures with the powerful tool of symbolic dynamics. We show in addition that an ergodic homoclinic class, admits a point with a forward orbit which is dense in a set which carries all conservative measures (possibly infinite) on it. More generally, an ergodic homoclinic class has a subset which carries all conservative measures on it, and which can be coded by an irreducible component.

\section{Definitions and Basic Properties}\text{ }

\medskip
This work uses tools which were previously developed in \cite{SBO,Sarig}. In the following subsection, we introduce two notions of hyperbolic points, in order to have a canonical\footnote{In this context, ``canonical" means definitions which do not rely on a specific construction of symbolic dynamics, but which depend only on the quality of hyperbolicity of the orbit of the point.} characterization for a set of points which our symbolic extension codes (see \cite{SBO,Sarig}).
\subsection{A Set of Hyperbolic and Weakly Temperable Points}

\begin{comment}The definition of $\chi$-hyperbolic points ($\chi>0$) in the context of Lyapunov regular points is quite natural- considering all points with no zero Lyapunov exponents, with at least one positive exponent and one negative exponent, and with all exponents having absolute value greater than $\chi$. The collection of points which display hyperbolic behavior can be a much bigger set than the Lyapunov regular points. %; and in addition analyzing which points in the image of the factor map (see \textsection\ref{symbolicdynamicspart}).
We consider a larger set of hyperbolic points. This serves two needs: one, we will have a somewhat reciprocal relation, where we can code our set of points with symbolic dynamics, while many of the chains in the symbolic space code points in our set. Two, a bigger and more inclusive set is a more natural candidate for a set to carry all hyperbolic measures which may also be infinite.
\end{comment}
Let $M$ be a compact Riemannian manifold without boundary, of dimension $d\geq 2$. Let $f\in\mathrm{Diff}^{1+\beta}(M), \beta>0$ (i.e. $f$ is invertible, $f, f^{-1}$ are differentiable, and both $d_\cdot f, d_\cdot f^{-1}$ are $\beta$-H\"older continuous). $\forall x\in M$, $\langle \cdot,\cdot\rangle_x:T_xM\times T_xM\rightarrow \mathbb{R}$ is the inner product on the tangent space of $x$ given by the Riemannian metric. $| \cdot |_x: T_xM\rightarrow\mathbb{R}$ is the norm induced by the inner product, $|\xi|_x^2:=\langle \xi,\xi\rangle_x$, $\forall \xi\in T_xM$. We often omit the $x$ subscript of the inner product and of the norm, when the tangent space in domain is clear by their argument.  

\noindent\textbf{Notations:}\begin{enumerate}
	\item For every $a,b\in\mathbb{R}$, $c\in\mathbb{R}^+$, $a=e^{\pm c}\cdot b$ means $e^{-c} \cdot b\leq a\leq e^c\cdot b$.
	\item $\forall a,b\in\mathbb{R}$, $a\wedge b:=\min\{a,b\}$.
	\item For every topological Markov shift $\Sigma$ induced by a graph $\mathcal{G}:=(\mathcal{V},\mathcal{E})$ (e.g. Theorem \ref{mainSBO}), $\forall v\in \mathcal{V}$, $[v]:=\{\ul{u}\in \Sigma: u_0=v\}$. 
\end{enumerate}

\begin{definition}\label{conservativeMeasures}
	Let $(X,\mathcal{B},\mu,T)$ be an invertible (perhaps infinite) measure preserving transformation. $\mu$ is said to be {\em conservative}, if it satisfies the statement of the Poincar\'e recurrence theorem. I.e.
	$$\forall A\in\mathcal{B},\text{ }\mu(A\setminus\{x\in A:\exists n_k,m_k\uparrow\infty\text{ s.t. }f^{n_k}(x), f^{-m_k}(x)\in A,\forall k\geq0\})=0.$$ 
\end{definition}

\begin{definition}\label{ChiHyp}\text{ }

\begin{enumerate}
         \item \begin{align*}
        \chi\mathrm{-summ}:=&\{x\in M: \exists\text{ a splitting }T_xM=H^s(x)\oplus H^u(x)\text{ s.t. }\\
        &\sup_{\xi_s\in H^s(x),|\xi_s|=1}\sum_{m=0}^\infty|d_xf^m\xi_s|^2e^{2\chi m}<\infty,\sup_{{\xi_u\in H^u(x),|\xi_u|=1}}\sum_{m=0}^\infty|d_xf^{-m}\xi_u|^2e^{2\chi m}<\infty\}.
    \end{align*}
    \item 
    \begin{align*}
    \chi\mathrm{-hyp}:=&\{x\in M: \exists\text{ a splitting }T_xM=H^s(x)\oplus H^u(x)\text{ s.t. }\forall\xi_s\in H^s(x)\setminus\{0\},\xi_u\in H^u(x)\setminus\{0\},\\ 
    &\limsup_{n\rightarrow\infty}\frac{1}{n}\log|d_xf^{n}\xi_s|,\limsup_{n\rightarrow\infty}\frac{1}{n}\log|d_xf^{-n}\xi_u|<-\chi\}.
    \end{align*}
    \item We define for each $x\in\chi\mathrm{-hyp}$, $$\chi(x):=-\max\{\sup_{\xi_s\in H^s(x)}\limsup_{n\rightarrow\infty}\frac{1}{n}\log|d_xf^{n}\xi_s|,\sup_{\xi_u\in H^u(x)}\limsup_{n\rightarrow\infty}\frac{1}{n}\log|d_xf^{-n}\xi_u|\}>\chi.$$
    \item $\forall x\in \chi\text{-summ}$, $\forall \xi,\eta\in T_xM$, write $\xi=\xi_s+\xi_u, \eta=\eta_s+\eta_u $ with $\xi_s,\eta_s\in H^s(x), \xi_u,\eta_u\in H^u(x)$, 
	\begin{align*}
		\langle \xi_s,\eta_s\rangle_{x,s}':=&2\sum_{m=0}^\infty\langle d_xf^m\xi_s, d_xf^m\eta_s\rangle_xe^{2\chi m},\\
			\langle \xi_u,\eta_u\rangle_{x,u}':=&2\sum_{m=0}^\infty\langle d_xf^{-m}\xi_u, d_xf^{-m}\eta_u\rangle_xe^{2\chi m},\\
			\langle \xi,\eta\rangle_{x}':=&\langle \xi_s,\eta_s\rangle_{x,s}'+ \langle \xi_u,\eta_u\rangle_{x,u}'.				
	\end{align*}
\end{enumerate}

\end{definition}

%\begin{equation}\label{summable} S^2(x,\xi_s):=2\sum_{k=0}^\infty|d_xf^k\xi_s|^2e^{2\chi k},U^2(x,\xi_u):=2\sum_{k=0}^\infty|d_xf^{-k}\xi_u|^2e^{2\chi k}.\end{equation}

%It is clear that every $\chi$-hyperbolic Lyapunov regular point is in this set, but also some Lyapunov regular points, with exponents exactly $\chi$- as long as the sub-exponential growth rate on orbits of tangent vectors is summable in the sense of \eqref{summable}.
Notice that $\chi\mathrm{-hyp}\subseteq\chi\mathrm{-summ}$. For each $x\in\chi\mathrm{-summ}$, write $s(x):=\mathrm{dim}(H^s(x)),u(x):=\mathrm{dim}(H^u(x))$. The following theorem is a version of the Pesin-Oseldec reduction theorem, which we prove in \cite[Theorem~2.4,Defintion.~2.5]{SBO}).
\begin{theorem}\label{defOfC}
$\forall x\in \chi$-summ, $\exists C_\chi(x):\mathbb{R}^d\rightarrow T_xM$ a linear invertible map, such that $\forall u,v\in\mathbb{R}^d$, $\langle C_\chi^{-1}(x)u, C_\chi^{-1}(x)v\rangle_x'=\langle u,v\rangle_2$, where $\langle\cdot ,\cdot\rangle_2$ is the Euclidean inner product on $\mathbb{R}^d$. %$C_\chi(x)$ is determined uniquely up to composition with an orthogonal map of the `` stable" and of the ``unstable" subspaces of the tangent spac; and can be chosen such that the map $x\mapsto C_\chi(x)$ is measurable. In addition, $C_\chi(x)[\mathbb{R}^{s(x)}\times\{0\}]=H^s(x),C_\chi(x)[\{0\}\times\mathbb{R}^{u(x)}]=H^u(x)$. $C_\chi(\cdot)$ are chosen measurably on $\chi\mathrm{-summ}$, and the choice is unique up to a composition with an orthogonal mapping of the `` stable" and of the ``unstable" subspaces of the tangent space. 
In addition,
\begin{equation}\label{train}
C_\chi^{-1}(f(x))\circ d_xf\circ C_\chi(x)=\begin{pmatrix}D_s(x)  &   \\  & D_u(x)
\end{pmatrix},
\end{equation}
where $D_s(x),D_u(x)$ are square matrices of dimensions $s(x),u(x)$ respectively, and $\|D_s(x)\|,\|D_u^{-1}(x)\|\leq e^{-\chi}$,$\|D_s^{-1}(x)\|,\|D_u(x)\|\leq \kappa$ for some constant $\kappa=\kappa(f,\chi)>1$.
\end{theorem}

\begin{claim}\label{contraction}
$\forall x\in \chi$-summ, $C_\chi(x)$ is a contraction.
\end{claim}
See \cite[Lemma~2.9]{SBO},\cite[Lemma~2.5]{Sarig}.

\begin{comment}
 there exists an invertible linear map $C_\chi(x):\mathbb{R}^d\rightarrow T_xM$, which depends measurably on $x$, such that %$\tilde{C}_\chi(x)[\mathbb{R}^{s(x)}\times\{0\}]=H^s(x),\tilde{C}_\chi(x)[\{0\}\times\mathbb{R}^{u(x)}]=H^u(x)$,
$C_\chi(x)[\mathbb{R}^{s(x)}\times\{0\}]=H^s(x),C_\chi(x)[\{0\}\times\mathbb{R}^{u(x)}]=H^u(x)$. % and $\|C_\chi^{-1}(x)\|<\|\tilde{C}_\chi^{-1}(x)\|$ for all $x\in \chi-\mathrm{hyp}$.%\footnote{The parallel use of two coordinates change is due to the need of points (weakly) temperable (see definition \ref{temperable}) w.r.t. $\tilde{C}_\chi(\cdot)$ (see definition \ref{temperable}), while using the more inclusive properties of $C_\chi(\cdot)$.} In addition, $$D_\chi(x):=C_\chi^{-1}(f(x)) \circ d_xf \circ C_\chi(x)$$ has the Lyapunov block form:
% With Respect to the standard basis
$C_\chi(\cdot)$ are chosen measurably on $\chi\mathrm{-summ}$, and the choice is unique up to a composition with an orthogonal mapping of the `` stable" and of the ``unstable" subspaces of the tangent space. In addition,
\begin{equation}\label{train}
C_\chi^{-1}(f(x))\circ d_xf\circ C_\chi(x)=\begin{pmatrix}D_s(x)  &   \\  & D_u(x)
\end{pmatrix},
\end{equation}
where $D_s(x),D_u(x)$ are square matrices of dimensions $s(x),u(x)$ respectively, and $\|D_s(x)\|,\|D_u^{-1}(x)\|\leq e^{-\chi}$,$\|D_s^{-1}(x)\|,\|D_u(x)\|\leq \kappa$ for some constant $\kappa=\kappa(f,\chi)>1$.
\end{comment}
\begin{definition}\label{canonicalHypQuant}
	Let $x\in\chi$-summ with $T_xM=H^s(x)\oplus H^u(x)$,
	\begin{enumerate}
		\item $$c_\chi(x):=\sup_{\xi_s\in H^s(x), \xi_u\in H^u(x):|\xi_s+\xi_u|=1}\sqrt{2\sum_{m\geq0}|d_xf^m\xi_s|^2e^{2\chi m}+2\sum_{m\geq0}|d_xf^{-m}\xi_u|^2e^{2\chi m}}.$$
		\item $\forall\xi_s\in H^s(x),\xi_u\in H^u(x)$, $$S^2(x,\xi_s):=2\sum_{m\geq0}|d_xf^m\xi_s|^2e^{2\chi m},\text{ }U^2(x,\xi_u):=2\sum_{m\geq0}|d_xf^{-m}\xi_u|^2e^{2\chi m}.$$
	\end{enumerate}
\end{definition}
$c_\chi(x)$ is a measurement of the hyperbolicity of $x$ w.r.t decomposition $T_xM=H^s(x)\oplus H^u(x)$- the greater it is, the worse the hyperbolicity (i.e slow contraction/expansion on stable/unstable tangent spaces, or small angle between the stable and unstable tangent spaces).

\begin{claim}\label{inversenorm}
	For $x\in \chi\mathrm{-summ}$, $\|C_\chi^{-1}(x)\|=c_\chi(x)$.%the norm of $C_\chi^{-1}(x)$ is given by
%	$$\|C_\chi^{-1}(x)\|=\sup_{\xi_s\in H^s(x),\xi_u\in H^u(x),|\xi_s+\xi_u|=1}\sqrt{S^2(x,\xi_s)+ U^2(x,\xi_u)},$$
%where
%$$S^2(x,\xi_s):=2\sum_{m\geq0}|d_xf^m\xi_s|^2e^{2\chi m}, U^2(x,\xi_u):=2\sum_{m\geq0}|d_xf^{-m}\xi_u|^2e^{2\chi m}.$$	
\end{claim}
This fact follows from the definition of the maps $C_\chi(\cdot)$, and can be seen in \cite[Theorem~2.4]{SBO} where they are defined. It is important to note that the expression for the norm of $C_\chi^{-1}$ depends only on the decomposition $H^s(x)\oplus H^u(x)$, independently of the choice of $C_\chi(\cdot)$. 

\medskip
%$\|C_\chi^{-1}(x)\|$ serves as a measurement of the hyperbolicity of $x$- the greater the norm, the worse the hyperbolicity (i.e slow contraction/expansion on stable/unstable tangent spaces, or small angle between the stable and unstable tangent spaces).% For points in $\chi$-summ which are weakly temperable (see Definition \ref{temperable}), we will be able to define a ``tempered" measurement of the hyperbolicity; in the sense of an average quality of hyperbolicty along their orbit. We use this measurement of hyperbolicity along the orbit of a hyperbolic point to define the set of ``regular" points (regular w.r.t. their hyperbolic behavior, not Oseledec theorem).

\begin{theorem}\label{mainSBO}
	 $\forall \chi>0$ s.t. $\exists p\in\chi$-$\mathrm{hyp}$ a periodic hyperbolic point, $\exists \epsilon_\chi>0$ (which depends on $M,f,\beta,\chi$) s.t. $\forall 0<\epsilon\leq \epsilon_\chi$ $\exists$ a countable and locally-finite directed graph $\mathcal{G}= \left(\mathcal{V},\mathcal{E}\right)= \left(\mathcal{V}(\epsilon), 
\mathcal{E}(\epsilon)\right)$ which induces a topological Markov shift $
\Sigma=\Sigma(\chi,\epsilon):=\{\ul{u}\in\mathcal{V}^\mathbb{Z}:
(u_i,u_{i+1})\in	\mathcal{E},\forall i\in\mathbb{Z}\}$. $\Sigma$ admits a factor map $\pi:\Sigma\rightarrow M$ with the following properties:
	 \begin{enumerate}
	 	\item $\sigma:\Sigma\rightarrow\Sigma$, $(\sigma \ul{u})_i:=u_{i+1}$, $i\in \mathbb{Z}$ (the left-shift); $\pi\circ\sigma=f\circ\pi$.
	 	\item $\pi$ is a H\"older continuous map w.r.t to the metric $d(\ul{u},\ul{v}):=\exp\left(-\min\{i\geq0:u_i\neq v_i\text{ or } u_{-i}\neq v_{-i}\}\right)$.
	 	\item $\Sigma^\#:=\left\{\ul{u}\in \Sigma:\exists n_k,m_k\uparrow\infty\text{ s.t. }u_{n_k}=u_{n_0}, u_{-m_k}=u_{-m_0},\forall k\geq0\right\}$, %$\pi|_{\Sigma^\#}$ is finite-to-one. \item 
	 $\pi[\Sigma^\#]$ carries all $f$-invariant, $\chi$-hyperbolic probability measures.\footnote{I.e. hyperbolic measures with Lyapunov exponents greater than $\chi$ in absolute value.}
	 \end{enumerate}
\end{theorem}
This theorem is the content of \cite[Theorem~3.13]{SBO} (and similarly, the content of \cite[Theorem~4.16]{Sarig} when $d=2$). $\mathcal{V}$ is a collection of double Pesin-charts (see %\cite[\textsection~2.3.1]{SBO}), 
Definition \ref{PesinCharts}), 
which is discrete.\footnote{\label{discreteness}Every $v\in\mathcal{V}$ is a double Pesin-chart of the form $v=\psi_x^{p^s,p^u}$ with $0<p^s,p^u\leq Q_\epsilon(x)$; and discreteness means that $\forall \eta>0:$ $\#\{v\in\mathcal{V}:v=\psi_x^{p^s,p^u} p^s\wedge p^u>\eta\}<\infty$.}

\begin{definition}\label{littleQ} Let $\epsilon>0$ and  $x\in\chi\mathrm{-summ}$, 

\medskip
%\begin{enumerate}
    %\item 
    $$Q_\epsilon(x):=\max\{Q\in \{e^{\frac{-\ell\epsilon}{3}}\}_{\ell\in\mathbb{N}}:Q\leq \frac{1}{3^\frac{6}{\beta}}\epsilon^\frac{90}{\beta}\left(c_\chi(x)\right)^\frac{-48}{\beta}\}.$$
%    $$\tilde{Q}_\epsilon(x):=\max\{Q\in \{e^{\frac{-\ell\epsilon}{3}}\}_{l\in\mathbb{N}}:Q\leq \frac{1}{3^\frac{6}{\beta}}\epsilon^\frac{90}{\beta}\|\tilde{C}^{-1}_\chi(x)\|^\frac{-48}{\beta}\}.$$
    %\item $$q_\epsilon(x):=\frac{\epsilon}{\sum_{n\in\mathbb{Z}}\frac{1}{Q_\epsilon(f^n(x))}e^{-|n|\epsilon}}.$$
%\end{enumerate}
%$q_\epsilon(x)$ is well defined by the temperedness of $x$, and fulfills $\frac{q_\epsilon\circ f}{q_\epsilon}=e^{\pm\epsilon}, q_\epsilon<\epsilon Q_\epsilon$.
%In particular, $\tilde{Q}_\epsilon(x)\leq Q_\epsilon(x)$, since $\|\tilde{C}_\chi^{-1}(\cdot)\|\geq\|C_\chi^{-1}(\cdot)\|$.
\end{definition}

\begin{definition}\label{PesinCharts} (Pesin-charts) Since $M$ is compact, $\exists r=r(M)>0$ s.t. the exponential map $\exp_x: \{v\in T_xM:|v|\leq r\}\rightarrow M$ is well defined and smooth. When $\epsilon\leq r$, the following is well defined since $C_\chi(\cdot)$ is a contraction (see Claim \ref{contraction}):
\begin{enumerate}
	\item $\psi_x^\eta:=\exp_x\circ C_\chi(x):\{v\in T_xM:|v|\leq \eta\}\rightarrow M$, $\eta\in (0,Q_\epsilon(x)]$, is called a {\em Pesin-chart}.
	\item A {\em double Pesin-chart} is an ordered couple $\psi_x^{p^s,p^u}:=(\psi_x^{p^s}, \psi_x^{p^u})$, where $\psi_x^{p^s}$ and $\psi_x^{p^u}$ are Pesin-charts. 
\end{enumerate} 
\end{definition}

\begin{definition}\label{temperable}
A point $x\in\chi\mathrm{-summ}$ is called {\em $\epsilon$-weakly temperable} if  $\exists q:\{f^n(x)\}_{n\in\mathbb{Z}}\rightarrow (0,\epsilon)\cap \{e^{\frac{-\ell\epsilon}{3}}\}_{l\in\mathbb{N}}$ s.t.
\begin{enumerate}
	\item $\frac{q\circ f}{q}=e^{\pm\epsilon}$,
	\item $\limsup\limits_{n\rightarrow\pm\infty}q(f^n(x))>0$,
	\item $\forall n\in\mathbb{Z}$, $q(f^n(x))\leq Q_\epsilon(f^n(x))$.
\end{enumerate}
The set of all $\epsilon$-weakly temperable points is denoted by {\em  $\epsilon$-w.t}.
%Whence $\lim_{n\rightarrow\pm\infty}\frac{1}{n}\log\|C^{-1}_\chi(f^n(x))\|=0$.
\end{definition}

\ul{\textbf{Remark:}}\begin{enumerate}
	\item Notice that item (3) could have been replaced by ``$\exists a>0$ s.t. $\forall n\in \mathbb{Z}$, $q(f^n(x))\leq \frac{a}{\left(c_\chi(f^n(x))\right)^\frac{48}{\beta}}$", since $\forall\epsilon>0\exists a_\epsilon>0$ s.t. $Q_\epsilon(\cdot)=a_\epsilon^{\pm1}\frac{1}{\left(c_\chi(\cdot)\right)^\frac{48}{\beta}}$, and $q$ can be rescaled. The $\frac{48}{\beta}$ exponent is not an intrinsic property- every fixed, sufficiently large, power of $c_\chi(\cdot)$ in the definition of $Q_\epsilon(\cdot)$ would have sufficed; altering the power alters both the set of hyperbolic points, and the coding, via the $\epsilon$-overlap condition, \cite[Definition~2.18]{SBO}. 
	\item In \cite[Claim~2.11]{SBO}, we show that $\forall \epsilon>0$, almost every point is $\epsilon$-weakly temperable w.r.t. every invariant probability measure carried by $\chi\mathrm{-summ}$, (by the fact that for almost every $x\in\chi$-summ, $\lim\limits_{n\rightarrow\pm\infty}\frac{1}{n}\log\|C_\chi^{-1}(f^n(x))\|=0$, and the construction of a Pesin's tempered kernel $q_\epsilon(x):=\frac{\epsilon}{\sum_{n\in\mathbb{Z}}\frac{1}{Q_\epsilon(f^n(x))}e^{-\frac{|n|}{3}\epsilon}}$; and finally $\limsup\limits_{n\rightarrow\pm\infty}q_\epsilon(f^n(x))>0$ for almost every point by Poincar\'{e}'s recurrence theorem).
\end{enumerate}
\begin{definition}\label{NUHsharp}\text{}
%
%\medskip
%\begin{enumerate}    \item
 $\HWT_\chi^\epsilon:= \chi\mathrm{-summ}\cap \epsilon\mathrm{-w.t}$, are the {\em hyperbolic and weakly temperable points}, with parameters $\chi,\epsilon>0$.
%    \item From this point onward, in our context, a {\em hyperbolic measure} (which is not necessarily finite) is a measure carried by $\bigcup\limits_{\chi'>0}\HWT^{\epsilon_{\chi'}}_{\chi'}$, and a {\em $\chi$-hyperbolic measure} is a measure carried by $\HWT_\chi^{\epsilon_\chi}$.

%\end{enumerate}
\end{definition}

\textbf{\underline{Remark:}} $Q_\epsilon(\cdot)$ depends only on $\epsilon$ and the norm of $C_\chi^{-1}(\cdot)$ (a Lyapunov norm on the tangent space of a point), %which is indifferent to composition with orthogonal mappings of the ``stable" and ``unstable" subspaces. 
which is given by Claim \ref{inversenorm}, and depends only on the decomposition $T_\cdot M= H^s(\cdot)\oplus H^u(\cdot)$.
By equation \eqref{train}, if $x\in \chi$-summ is also $\epsilon$-weakly temperable (and $\epsilon$ is small w.r.t $\chi,\beta$, as imposed by the assumption $\epsilon\leq\epsilon_\chi$ from Theorem \ref{mainSBO}), then the decomposition $T_xM=H^s(x)\oplus H^u(x)$ must be unique. 
Therefore, $Q_\epsilon(\cdot)$ is defined canonically on $\epsilon$-w.t, and does not depend on the choice of $C_\chi(\cdot)$.
Thus, $\forall \epsilon\in (0,\epsilon_\chi]$, $\HWT_\chi^\epsilon$ is defined canonically.

\textbf{\underline{Remark:}}
$\HWT_\chi^\epsilon$ is of full measure w.r.t. every invariant probability measure carried by $\chi\mathrm{-summ}$. In \cite{SBO} (and in \cite{Sarig}), the authors construct a Markov partition and show that it covers a smaller set than $\HWT_\chi^\epsilon$ %from Definition \ref{NUHsharp} (although the smaller set is also referred to as $NUH_\chi^\#$ in both \cite{SBO,Sarig}; this is addressed in the remark after the statement of Theorem \ref{BMS})
. Theorem \ref{BMS} shows that the same construction in \cite{SBO} can be done to cover all of $\HWT_\chi^\epsilon$ from Definition \ref{NUHsharp}, $\forall \epsilon\in (0,\epsilon_\chi]$.
%and that every chain in the symbolic space which involves only a finite number of symbols, has an image in $NUH_\chi^\#$. A heuristic for that fact is that $NUH_\chi^\#$ are points with a nice hyperbolic behavior along their orbits, and chains with a finite amount of symbols in them code the orbits of points with a uniformly-hyperbolic orbit.

\begin{claim}\label{canonicalepsilon}
	$\forall \epsilon>0$ and $\epsilon'\geq \frac{3}{2}\epsilon$, $\HWT_\chi^\epsilon\subseteq \HWT_\chi^{\epsilon'}$.
\end{claim}
\begin{proof}
Let $x\in \HWT_\chi^\epsilon$, and let $q:\{f^n(x)\}_{n\in\mathbb{Z}}\rightarrow (0,\epsilon)\cap\{e^{-\frac{\epsilon\ell}{3}}\}_{\ell\geq0}$ be given by the $\epsilon$-weak temperability of $x$.	Define $\widetilde{q}(f^n(x)):=\max\{t\in \{e^{-\frac{\epsilon'\ell}{3}}\}_{\ell\geq0} : t\leq q(f^n(x))\}$. It follows that $\frac{\widetilde{q}_r\circ f}{\widetilde{q}}=e^{\pm(\epsilon+\frac{\epsilon'}{3})}=e^{\pm(\frac{\epsilon'}{3}+ \frac{2\epsilon'}{3})}=e^{\pm\epsilon'}$. It follows from Definition \ref{littleQ} that $\exists \widetilde{b}(\epsilon,\epsilon')>0$ s.t. $\forall n\in\mathbb{Z}$, $\widetilde{b}(\epsilon,\epsilon')\cdot Q_\epsilon(f^n(x))\leq Q_{\epsilon'}(f^n(x))$. Let $b(\epsilon,\epsilon'):= \max\{t\in \{e^{-\frac{\epsilon'\ell}{3}}\}_{\ell\geq0} : t\leq \widetilde{b}(\epsilon,\epsilon')\}$, and define $q'(f^n(x)):=b(\epsilon,\epsilon')\cdot \widetilde{q}(f^n(x))$, $n\in\mathbb{Z}$. Since $\{e^{-\frac{\epsilon'\ell}{3}}\}_{\ell\geq0}$ is closed under multiplication, it follows that $q'$ satisfies the assumptions of $\epsilon'$-weak temperability for $x$, and so $x\in \HWT_\chi^{\epsilon'}$. 
\end{proof}

\section{Symbolic Dynamics}\label{symbolicdynamicspart}
We now present some changes to the construction of $\epsilon_\chi,\mathcal{V},\Sigma$, and in Theorem $\ref{BMS}$ we will show that this does not affect the statements of \cite{SBO}. On the other hand, these changes will allow us to construct the symbolic dynamics in such a way that we could characterize the image of $\Sigma^\#$ (for every $\epsilon>0$ small enough).

\medskip
Assume that there exists a $\chi$-hyperbolic periodic point $p$, so that $\exists \epsilon_\chi>0$ as in Theorem \ref{mainSBO}.
\begin{definition}\label{changesToCoding}\text{ }

 \begin{enumerate}
     \item $\epsista_\chi:= \min\{\epsilon_{\frac{\chi}{2}},\epsilon_\chi\}>0$.
     \item $\Vista_\epsilon$, $\epsilon\in(0,\epsista_\chi]$ is a collection of double Pesin-charts %(\cite[\textsection~2.3.1]{SBO}%,\cite[\textsection~4.1]{Sarig}), for  
     in the set $\Aista_\epsilon$, where $\Aista_\epsilon$ is constructed in the Coarse Graining process for $\chi\mathrm{-summ}\cap\epsilon\mathrm{-w.t}$ (instead of Pesin-charts with centers in $NUH^*_\chi$, see \cite[Proposition~2.22]{SBO},\cite[Proposition~3.5]{Sarig}).
     \item A vertex $v\in \Vista_\epsilon$ is called {\em $\epsilon$-relevant}, if $\exists \ul{u}\in \Sista_\epsilon\cap[v]$ s.t. $\pi(\ul{u})\in \HWT_\chi^\epsilon$ (instead of the previous definition, \cite[Definition~3.14]{SBO}).\footnote{$\pi:\Sista_\epsilon\rightarrow M$ is defined by the Graph Transform (\cite[Proposition~3.12]{SBO},\cite[Proposition~4.12]{Sarig}), the same way as $\pi:\Sigma\rightarrow M$, and satisfies the properties from the statement of Theorem \ref{mainSBO} ($\forall i\in\mathbb{Z}$, $u_i=\psi_{x_i}^{p^s_i,p^u_i}$ satisfies equation \eqref{train}, while the assumptions for $\psi_{x_i}^{p^s_i\wedge p_i^u}\rightarrow \psi_{x_{i+1}}^{p^s_{i+1}\wedge p_{i+1}^u}$ remain unchanged).}
     \item $\Eista_\epsilon\subseteq \Vista_\epsilon\times\Vista _\epsilon $ is a set of edges, characterized by the same $\epsilon$-overlap condition with no change (see \cite[\textsection~3.0.2,Definition~2.23,Definition~2.18]{SBO}).
     \item $\Gista_\epsilon=(\Vista_\epsilon,\Eista_\epsilon)$ is a countable locally-finite directed graph (the local-finiteness of $\Gista_\epsilon$ follows from the discreteness of $\Vista_\epsilon$, see footnote \ref{discreteness})%\footnote{\label{discreteness}Every $v\in\Vista_\epsilon$ is a double chart of the form $v=\psi_x^{p^s,p^u}$ where $0<p^s,p^u\leq Q_\epsilon(x)$ are window parameters (see \cite[\textsection~2.3.1]{SBO}) and $\psi_x^{p^s\wedge p^u}\in \Aista_\epsilon$; and discreteness is in the sense that $\forall \eta>0:$ $\#\{v\in\Vista_\epsilon:v=\psi_x^{p^s,p^u} p^s\wedge p^u>\eta\}<\infty$.}
.
     \item $\Sista_\epsilon:=\{\ul{u}\in \overset{\star}{\mathcal{V}}\mathmiddlescript{{\epsilon^{^\mathcal{Z}}}}\text{ s.t. }(u_i,u_{i+1})\in\Eista_\epsilon,\forall i\in\mathbb{Z}\}$ is the topological Markov shift induced by $\Gista_\epsilon$.
     \item $\Sistasharpepsilon:=\{\ul{u}\in\Sista_\epsilon:\exists n_k,m_k\uparrow\infty\text{ s.t. }u_{n_k}=u_{n_0}, u_{-m_k}=u_{-m_0},\forall k\geq0\}$.

 \end{enumerate}
\end{definition}

Without loss of generality, we dismiss all vertices in $\Vista_\epsilon$ which are not $\epsilon$-relevant, and assume that every $v\in \Vista_\epsilon$ is $\epsilon$-relevant.

\begin{lemma}\label{newHope}
 Let $\epsilon\in (0,\epsista_\chi]$, and let $\ul{u}\in\Sistasharpepsilon$ be an admissible chain of double Pesin-charts. %Let $\ul{u}$ be the admissible periodic concatenation of $(u_0,\cdots,u_0)$ to itself.
 Then $p:=\pi(\underline{u})\in \chi\mathrm{-summ}$.
\end{lemma}
\begin{proof}
Write $u_i=\psi_{x_i}^{p^s_i,p^u_i}$,$i\in\mathbb{Z}$. For each $i\in\mathbb{Z}$, $x_i\in\chi\mathrm{-summ}$, whence, in particular, $x_i\in r\mathrm{-hyp}$ $\forall r\in [\frac{\chi}{2},\chi)$. %Notice that in \cite{SBO}, w.l.o.g. $t\leq s\Rightarrow \epsilon_t\leq \epsilon_s$. Therefore, by the restriction $\epsilon_\chi\mapsto\epsilon_\frac{\chi}{2}$ from Definition \ref{changesToCoding}(1), 
By the definition of $\epsista_\chi$, since $\epsilon\leq \epsista_\chi$, $Q_{\epsilon}(x_i)$ is small enough for the Graph Transform with the Lyapunov change of coordinate $C_r(\cdot)$, $\forall r\in[\frac{\chi}{2},\chi)$ (see \cite[Theorem~3.6,\textsection~2.1.2]{SBO}). We prove that $\sup_{\xi_s\in T_pV^s(\ul{u}),|\xi_s|=1}\sum_{m=0}^\infty|d_pf^m\xi_s|^2e^{2\chi m}<\infty$, the case for $H^u(p)$ is similar. W.l.o.g., $u_{n_k}=u_0$ for $n_k\uparrow\infty$.

Step 1: By the relevance of $u_0$, take some chain $\ul{w}\in \Sista_\epsilon\cap [u_0]$ s.t. $z:=\pi(\ul{w})\in \HWT_\chi^\epsilon=\chi\mathrm{-summ}\cap\epsilon$-w.t, whence in $r\mathrm{-hyp}$ $\forall r\in[\frac{3\chi}{5},\chi)$. In the proof of \cite[Lemma~4.5]{SBO}, temperability can be replaced by $\epsilon$-weak temperability (for $0<\epsilon\leq\epsista_\chi$), a more relaxed assumption. It follows then, that $\forall r\in[\frac{3\chi}{5},\chi)$ $\exists C=C(z,\frac{3\chi}{5},r)$, s.t. $\forall y\in V^s(\ul{w})$, $\sup_{\xi_s\in T_yV^s(\ul{w}),|\xi_s|=1}\sum_{m=0}^\infty|d_yf^m\xi_s|^2e^{2(r-\frac{\chi-r}{4}) m}\leq C<\infty$. Now, since $r-\frac{\chi-r}{4}\geq\frac{\chi}{2}$, we are free to use \cite[Lemma~4.6]{SBO} and claim 2 in \cite[Lemma~4.7]{SBO}, and get 
\begin{equation}\label{Lauranou}\sup_{\xi_s\in T_pV^s(\ul{u}),|\xi_s|=1}\sum_{m=0}^\infty|d_pf^m\xi_s|^2e^{2(r-\frac{\chi-r}{4}) m}\leq e^{\epsilon_\chi^\frac{1}{2}}\cdot\sup_{\xi_s\in H^s(x_0),|\xi_s|=1}\sum_{m=0}^\infty|d_{x_0}f^m\xi_s|^2e^{2(r-\frac{\chi-r}{4}) m}\leq \frac{5}{4}\|C_\chi^{-1}(x_0)\|^2.\end{equation}
The last inequality is true since $t\mapsto \|C_t^{-1}(x)\|$ is monotonous for every $x\in\chi\mathrm{-summ}$. Thus, $p\in r'\mathrm{-hyp}$ $\forall r'<\chi$.

Step 2: Fix $\xi_s\in T_pV^s(\ul{u}), |\xi_s|=1$. If $\sum_{m=0}^\infty|d_pf^m\xi_s|^2e^{2\chi m}\geq 2\|C_\chi^{-1}(x_0)\|^2$, then choose $N>0$ s.t. $\sum_{m=0}^N|d_pf^m\xi_s|^2e^{2\chi m}\geq \frac{3}{2}\|C_\chi^{-1}(x_0)\|^2$. Choose $r'<\chi$ s.t. $\sum_{m=0}^N|d_pf^m\xi_s|^2e^{2r' m}\geq \frac{4}{3}\|C_\chi^{-1}(x_0)\|^2$, whence $\sum_{m=0}^\infty|d_pf^m\xi_s|^2e^{2r' m}\geq \frac{4}{3}\|C_\chi^{-1}(x_0)\|^2$, a contradiction to equation \ref{Lauranou} from step 1! Therefore,
 \begin{equation}\label{finishhim}\forall \xi_s\in T_pV^s(\ul{u}), |\xi_s|=1,\sum_{m=0}^\infty|d_pf^m\xi_s|^2e^{2\chi m}\leq 2\|C_\chi^{-1}(x_0)\|^2,\end{equation}
 and similarly with $T_pV^u(\ul{u})$; and so $p\in \chi\mathrm{-summ}$. 
\end{proof}
\begin{theorem}\label{BMS}$\forall \epsilon\in(0,\epsista_\chi]$, $\pi[\Sistasharpepsilon]=\HWT_\chi^\epsilon$.
\end{theorem}
In \cite[Definition~2.17]{SBO} (and similarly in \cite[\textsection~2.5]{Sarig} when $d=2$) the authors offer a set which is covered by the constructed Markov partition $\mathcal{R}$, and denote this set by $NUH_\chi^\#$. %Those definitions are different from the one offered in definition \ref{NUHsharp}, and are more strict. These definitions involve Lyapunov regularity, and we claim that this was unnecessary, when wishing to work with conservative (perhaps infinite, $\sigma$-finite) measures.
The definition of $NUH_\chi^\#$ involves Lyapunov regularity, and we claim that this was unnecessary, when wishing to work with conservative (perhaps infinite) measures. %We offer our more relaxed Definition \ref{NUHsharp}, which we show in the proof below to still admit the property of being coded by a Markov partition
In Definition \ref{NUHsharp} we introduce a definition for a more inclusive set $\HWT_\chi^\epsilon$, which does not depend on Lyapunov regularity. We show in the proof below that $\HWT_\chi^\epsilon$ still admits the property of being coded by a Markov partition; but in fact it carries additional natural properties which we will see later (e.g. Corollary \ref{cafeashter}, Proposition \ref{homoclinicirreducible}). %We choose to name the more inclusive set which is being coded in this paper by the same name, $NUH_\chi^\#$, as we consider it a more natural and correct object in our setup, and not just an alternative object; we wished to emphasize that by reassigning the notation.
\begin{proof}
Let $\epsilon\in(0,\epsista_\chi]$. The following steps are done both in \cite{SBO}, and in \cite{Sarig} when $d=2$. We follow each step and give references to \cite{SBO} for the case $d\geq2$, although these references are analogues to the preceding work in \cite{Sarig}.

Step 1: In \cite[Definition~2.10]{SBO}, the authors introduce $NUH_\chi^*$, a set of Lyapunov regular $\chi$-hyperbolic points, such that $\forall x\in NUH_\chi^* ,\lim_{n\rightarrow\pm\infty}\frac{1}{n}\log\|C_\chi^{-1}(f^n(x))\|=0$ (\cite[Claim~2.11]{SBO}). This allows us to define Pesin's tempered kernel $\forall x\in NUH_\chi^*$, $q_\epsilon(x):=\frac{\epsilon}{\sum_{n\in\mathbb{Z}}\frac{1}{Q_\epsilon(f^n(x))}e^{-\frac{|n|}{3}\epsilon}}$, and so allows to define $NUH_\chi^\#:=\{x\in NUH_\chi^*: \limsup_{n\rightarrow\pm\infty}q_\epsilon(f^n(x))>0\}$ as well (\cite[Definition~2.17]{SBO}).

Step 2: By \cite[Proposition~2.22]{SBO}, there exists a discrete and sufficient collection of Pesin-charts with centers in $NUH_\chi^*$, $\mathcal{A}$; as in Definition \ref{changesToCoding}, $\Aista_\epsilon$ is constructed similarly, over a collection of Pesin-charts with centers in $\HWT_\chi^\epsilon$. 

Step 3: $\HWT_\chi^\epsilon\subseteq \epsilon$-w.t, and so an analogue to Pesin's tempered kernel exists by definition $\forall x\in \HWT_\chi^\epsilon$. Consequently, as in \cite[Proposition~2.30]{SBO}, $\pi[\Sista_\epsilon]\supseteq \HWT_\chi^\epsilon$. In addition, by \cite[Theorem~3.13(3), Proposition~3.16]{SBO}, $\pi[\Sistasharpepsilon]\supseteq \HWT_\chi^\epsilon$.

Step 4: Let $\ul{u}\in\Sistasharpepsilon$, and write $z:=\pi(\ul{u})$. W.l.o.g. write $u_{n_k}=u_0$, $\forall k\geq0$, and $n_k\uparrow\infty$. Consider the finite periodic word $\ul{w}=(u_0,u_1,...,u_{n_1-1},u_0)$, and associate it with its periodic extension to a chain. Consider the chains $\ul{u}^{(l)}\in \Sistasharpepsilon$, $l\geq0$, where $u^{(l)}_i = \Big\{\begin{array}{lr}
        u_i, & \text{for } i\leq n_l\\
        w_{i-n_l}, & \text{for } i\geq n_l\\
        \end{array}$. Write $z_l:=\pi(\ul{u}^{(l)})$. In \cite[Lemma~4.7]{SBO}, the author uses \cite[Lemma~4.5]{SBO} in order to show that $\sup_{\l \geq0}\sup_{\xi\in T_{z_l}V^s(\ul{u}^{(l)}),|\xi|=1}S(z_l,\xi)<\infty$; equation \eqref{finishhim} in Lemma \ref{newHope} takes the place of \cite[Lemma~4.5]{SBO} in this argument. It follows that \cite[Lemma~4.6,Lemma~4.7(Claim 2)]{SBO} can be carried out verbatim, and so $\exists$ a linear invertible map $\pi_{x_0}^s:T_zV^s(\ul{u})\rightarrow H^s(x_0)$ s.t. $\|\pi_{x_0}^s\|, \|\left(\pi_{x_0}^s\right)^{-1}\|\leq e^{2Q_\epsilon(x_0)^{\frac{\beta}{4}}}$, and
        \begin{align*}
				\forall \xi\in T_{z}V^s(\ul{u}),|\xi|=1,S(z,\xi)=e^{\pm\sqrt\epsilon}S(x_0,\pi_{x_0}^s\xi).        
        \end{align*}
A similar statement holds for $\pi_{x_0}^u:T_zV^u(\ul{u})\rightarrow H^u(x_0)$. $\pi_{x_0}^s$ and $\pi_{x_0}^u$ extend to the invertible linear map $\pi_{x_0}:T_zM\rightarrow T_{x_0}M$ by $\pi_{x_0}|_{T_zV^s(\ul{u})}=\pi_{x_0}^s$ and $\pi_{x_0}|_{T_zV^u(\ul{u})}=\pi_{x_0}^u$. In particular, $z\in \chi$-summ. It then follows from the proof of \cite[Proposition~4.8]{SBO} (though not specified) that $\frac{\|C_\chi^{-1}(z)\|}{\|C_\chi^{-1}(x_0)\|}=e^{\pm(4\sqrt\epsilon+\epsilon)}= e^{\pm5\sqrt\epsilon}$.

Step 5: By Lemma \ref{newHope}, $\forall \ul{u}\in \Sistasharpepsilon$, $\pi(\ul{u})\in \chi$-summ. In addition, $q(f^n(\pi(\ul{u}))):=b_\epsilon\cdot p^s_n\wedge p^u_n$, where $b_\epsilon:=\max\{t\in\{e^{-\frac{\ell\epsilon}{3}}\}_{\ell\geq0}:t\leq e^{-\frac{300\sqrt\epsilon}{\beta}}\}$ and $u_n=\psi_{x_n}^{p^s_n,p^u_n}$. By the definition of $\Aista_\epsilon$ (and so $\Vista_\epsilon$), $p_n^s\wedge p^u_n\in \{e^{-\frac{\ell\epsilon}{3}}\}_{\ell\geq0}$, $\forall n\in\mathbb{Z}$; thus, since $\{e^{-\frac{\ell\epsilon}{3}}\}_{\ell\geq0}$ is closed under multiplication, $q:\{f^n(\pi(\ul{u}))\}_{n\in\mathbb{Z}}\rightarrow (0,\epsilon)\cap \{e^{-\frac{\ell\epsilon}{3}}\}_{\ell\geq0}$. $q$ satisfies the assumptions of $\epsilon$-weak temperability: 
\begin{enumerate}
\item By the $\epsilon$-overlap condition, \cite[Definition~2.18]{SBO}, $\frac{q\circ f}{q}=e^{\pm\epsilon}$.
\item Since $\ul{u}\in\Sistasharpepsilon$, $\limsup\limits_{n\rightarrow\pm\infty} q\circ f^n(x)>0$.
\item  By the definition of double Pesin-charts, % \cite[\textsection~2.3.1]{SBO},
  $q(f^i(\pi(\ul{u})))\leq Q_\epsilon(x_i)$, $\forall i\in\mathbb{Z}$. Thus, $\forall n\in\mathbb{Z}$,

 by step 4 (recall Definition \ref{littleQ}, for definition of $Q_\epsilon$),
\begin{align*}p_n^s\wedge p_n^u\leq& Q_\epsilon(x_n)\leq \frac{\epsilon^\frac{90}{\beta}}{3^\frac{6}{\beta}}\|C_\chi^{-1}(x_n)\|^\frac{-48}{\beta} \frac{\epsilon^\frac{90}{\beta}}{3^\frac{6}{\beta}}\left(e^{5\sqrt\epsilon}\|C_\chi^{-1}(f^n(\pi(\ul{u})))\|\right)^\frac{-48}{\beta}\\
\leq &e^{5\sqrt\epsilon\cdot\frac{48}{\beta}}\cdot Q_\epsilon(f^n(\pi(\ul{u})))\cdot e^{\frac{\epsilon}{3}}\leq e^{\frac{300\sqrt\epsilon}{\beta}}\cdot Q_\epsilon(f^n(\pi(\ul{u}))).	
\end{align*}
Therefore $q(f^n(\pi(\ul{u})))\leq Q_\epsilon(f^n(\pi(\ul{u})))$.
\end{enumerate}
It follows that $\pi(\ul{u})\in\chi\mathrm{-summ}\cap\epsilon\mathrm{-w.t}=\HWT_\chi^\epsilon$. Thus, $\pi[\Sistasharpepsilon]\subseteq \HWT_\chi^\epsilon$, and together with step 3, $\pi[\Sistasharpepsilon]=\HWT_\chi^\epsilon$.

\end{proof}

\begin{cor}\label{localfinito}
 $\forall \epsilon\in (0,\epsista_\chi]$, $\forall u\in \Vista_\epsilon$, $\#\{v\in \Vista_\epsilon: \Zista_\epsilon(u)\cap\Zista_\epsilon(v)\neq\varnothing\}<\infty$, where $\Zista_\epsilon(v):=\pi[\Sistasharpepsilon\cap[v]]$.
\end{cor}
This is the content of \cite[Theorem~5.2]{SBO} (and similarly \cite[Theorem~10.2]{Sarig} when $d=2$), where (as in step 4 in Theorem \ref{BMS}) Lemma \ref{newHope} replaces \cite[Lemma~4.5]{SBO} in the proof \cite[Lemma~4.7]{SBO}; and the rest of the Inverse Problem (\cite[\textsection~4]{SBO}) can be carried out verbatim with $\Sista_\epsilon$ replacing $\Sigma$.

\medskip
Assume that there exists a periodic point $p\in\chi$-hyp, and let $\epsilon\in(0,\epsista_\chi]$. Let $\epsilon\in(0,\epsista_\chi]$.
\begin{definition}\label{Doomsday}
\text{ }

\medskip
\begin{enumerate}
	\item $\forall u\in \Vista_\epsilon$, $\Zista_\epsilon(u)=\pi[[u]\cap\Sistasharpepsilon]$, $\mathcal{Z}_\epsilon:=\{\Zista_\epsilon(u):u\in\Vista_\epsilon\}$.
	\item $\mathcal{R}_\epsilon$ is a countable partition of $\bigcup\limits_{v\in\Vista_\epsilon}\Zista_\epsilon(v)=\pi[\Sistasharpepsilon]$, s.t.
	\begin{enumerate}
		%\item $\forall Z\in\mathcal{Z}_\epsilon$,$\#\{Z'\in\mathcal{Z}_\epsilon:Z'\cap Z\neq\varnothing\}<\infty$ (Corollary \ref{localfinito}).
		\item $\mathcal{R}_\epsilon$ is a refinement of $\mathcal{Z}_\epsilon$: $\forall Z\in\mathcal{Z}_\epsilon,R\in\mathcal{R}_\epsilon$, $R\cap Z\neq\varnothing\Rightarrow R\subseteq Z$.
		\item $\forall v\in\Vista_\epsilon$, $\#\{R\in\mathcal{R}_\epsilon:R_\epsilon\subseteq \Zista_\epsilon(v)\}<\infty$ (\cite[\textsection~11]{Sarig}).
		\item The Markov property: $\forall R\in\mathcal{R}_\epsilon$,$\forall x,y\in R$ $\exists ! z:=[x,y]_R\in R$, s.t. $\forall i\geq0, R(f^i(z))= R(f^i(y)), R(f^{-i}(z))= R(f^{-i}(x))$, where $R(t)$ is the unique partition member of $\mathcal{R}_\epsilon$ containing $t$, for $t\in\pi[\Sistasharpepsilon]$.
	\end{enumerate}
	\item $\forall R,S\in\mathcal{R}_\epsilon$, we say $R\rightarrow S$ if $R\cap f^{-1}[S]\neq\varnothing$, i.e. $\widehat{\mathcal{E}}_\epsilon=\{(R,S)\in\mathcal{R}_\epsilon^2\text{ s.t. }f^{-1}[S]\cap R\neq\varnothing\}$.
	\item $\Sig_\epsilon:=\{\ul{R}\in\mathcal{R}_\epsilon^{\mathbb{Z}}: R_i\rightarrow R_{i+1},\forall i\in\mathbb{Z}\}$.
\end{enumerate}	
\end{definition}
Given $\mathcal{Z}_\epsilon$, such a refining partition as $\mathcal{R}$ exists by the Bowen-Sinai refinement, see \cite[\textsection~11.1]{Sarig}.
\begin{definition}\label{sigmasharp}\text{ }

\medskip
\begin{enumerate}
    \item $\Sig^\#_\epsilon:=\{\ul{R}\in\Sig_\epsilon:\exists n_k,m_k\uparrow\infty\text{ s.t. }R_{n_k}=R_{n_0},R_{-m_k}=R_{-m_0},\forall k\geq0\}$.
%    \begin{align*}
%        &\Sigma^\#:=\{\ul{u}\in\Sigma:\exists n_k,m_k\uparrow\infty\text{ s.t. }u_{n_k}=u_{n_0},u_{-m_k}=u_{-m_0}\forall k\geq0\},\\
%        &
%        \end{align*}
    \item Every two partition members $R,S\in \mathcal{R}_\epsilon$ %(where $\mathcal{R}_\epsilon$ is the Markov partition  given by the Bowen-Sinai refinement of the cover $\mathcal{Z}_\epsilon$, see \cite[proof of Theorem 1.1]{SBO}), 
are said to be {\em $\epsilon$-affiliated} if $\exists u,v\in\Vista_\epsilon$ s.t. $R\subseteq \Zista_\epsilon(u), S\subseteq \Zista_\epsilon(v)$ and $\Zista_\epsilon(u)\cap \Zista_\epsilon(v)\neq\varnothing$ (this is due to O. Sarig, \cite[\textsection~12.3]{Sarig}).
%    \item Every sequence (finite or infinite) of symbols in $\Vista_\epsilon$ is called {\em admissible} if it can be extended to a chain in $\Sigma_\epsilon$. Similarly for sequences of symbols in $\mathcal{R}$, and $\Sig$. Chains in $\Sigma,\Sig$ are also sometimes referred to as admissible.
\end{enumerate}
\end{definition}
\textbf{\underline{Remark:}} %An addition to the remark after definition \ref{NUHsharp}, is that the Markov partition which covers $NUH_\chi^\#$, is a refinement of $\mathcal{Z}$. Therefore, $NUH_\chi^\#\subseteq \pi[\Sigma^\#]$. In addition, 
By Corollary \ref{localfinito} and Definition \ref{Doomsday}(2)(b), it follows that every partition member of $\mathcal{R}_\epsilon$ has only a finite number of partition members $\epsilon$-affiliated to it. 
\begin{theorem}
	Given $\Sig_\epsilon$ from Definition \ref{Doomsday}, there exists a factor map $\widehat{\pi}:\Sig_\epsilon\rightarrow M$ s.t.
	\begin{enumerate}
		\item $\widehat{\pi}$ is H\"older continuous w.r.t the metric $d(\ul{R},\ul{S})=\exp\left(-\min\{i\geq0: R_i\neq S_i\text{ or }R_{-i}\neq S_{-i}\}\right)$.
		\item $f\circ\widehat{\pi}=\widehat{\pi}\circ\sigma$, where $\sigma$ denotes the left-shift on $\Sig_\epsilon$.
		\item $\widehat{\pi}|_{\Sig_\epsilon^\#}$ is finite-to-one.
		\item $\forall \ul{R}\in \Sig_\epsilon$, $\widehat{\pi}(\ul{R})\in \overline{R_0}$.
		\item $\widehat{\pi}[\Sig_\epsilon^\#]$ carries all $\chi$-hyperbolic invariant probability measures.
	\end{enumerate}
\end{theorem}
This theorem is the content of the main theorem of \cite{SBO}, Theorem 1.1 (and similarly the content of \cite[Theorem~1.3]{Sarig} when $d=2$). %Notice that $\Sig$ depends on $\Sigma$, and so on $\epsilon$ which was fixed in the beginning of the section.

\begin{prop}\label{imagecanonic} $\forall \epsilon\in (0,\epsista_\chi]$,
$$\widehat{\pi}[\Sig_\epsilon^\#]=\pi[\Sistasharpepsilon]=\bigcupdot\mathcal{R}_\epsilon.$$
\end{prop}
\begin{proof}
$\pi[\Sistasharpepsilon]=\bigcupdot\mathcal{R}_\epsilon$ by definition, since $\mathcal{R}_\epsilon$ is a partition of $\pi[\Sistasharpepsilon]$. We need to show $\widehat{\pi}[\Sig_\epsilon^\#]=\pi[\Sistasharpepsilon]$. 

\underline{$\supseteq$:} Let $\ul{u}\in \Sistasharpepsilon$, $\pi(\ul{u})\in \mathcal{R}_\epsilon$. Write $R_i:=$unique element of $\mathcal{R}_\epsilon$ which contains $f(\pi(\ul{u}))$, $i\in\mathbb{Z}$. It follows that $\ul{R}:=(R_i)_{i\in\mathbb{Z}}\in \Sig_\epsilon$, and that $\widehat{\pi}(\ul{R})=\pi(\ul{u})$ by the uniqueness of a shadowed orbit. Then by definition, $\forall i\in\mathbb{Z}$, $R_i\subseteq \Zista_\epsilon(u_i)$, and so by the pigeonhole principle, $\ul{R}\in \Sig_\epsilon^\#$ (see Definition \ref{Doomsday}(2)(b)).% (see the remark following Definition \ref{sigmasharp}).

\underline{$\subseteq$:} Let $R,S\in \mathcal{R}_\epsilon$ s.t. $\exists x\in R\cap f^{-1}[S]$. Let $u\in\Vista_\epsilon$ s.t. $R\subseteq \Zista_\epsilon(u)$. Then $\exists \ul{u}\in \Sistasharpepsilon\cap[u]$ s.t. $\pi(\ul{u})=x$, and so $S\subseteq \Zista_\epsilon(u_1)$. Given a chain $\ul{R}\in \Sig_\epsilon^\#$, choose $\Zista_\epsilon(u_0)\supseteq R_0$, and extend it this way to a chain $\ul{u}\in \Sista$ s.t. $R_i\subseteq \Zista_\epsilon(u_i)$ $\forall i\in\mathbb{Z}$. By the uniqueness of a shadowed orbit, $\pi(\ul{u})=\widehat{\pi}(\ul{R})$. By Corollary \ref{localfinito}%item (1)(a) in Definition \ref{Doomsday}
, and the pigeonhole principle, $\ul{u}\in \Sistasharpepsilon$. 
%In definition \ref{canonicparts} we've seen $\Sig^\circ\subseteq\Sig^\#$. In order to see that $\widehat{\pi}[\Sig^\#]\subseteq \pi[\Sigma^\#]$, for a chain $\ul{R}\in\Sig^\#$, take some $\ul{u}\CAR\ul{R}$. By local finiteness of the refinement $\ul{u}\in \Sigma^\#$ and $\widehat{\pi}(\ul{R})=\pi(\ul{u})\in\pi[\Sigma^\#]$. Finally, $\pi[\Sigma^\#]\subseteq \widehat{\pi}[\Sig^\circ]$ because for any $x\in\pi[\Sigma^\#]=\bigcupdot\mathcal{R}$, $\ul{R}(x)\in\Sig^\circ$ and $\widehat{\pi}(\ul{R}(x))=x$. %Finally, if $\ul{R}\in\Sig^\circ$ then by definition $\widehat{\pi}(\ul{R})\in R_0\subseteq\bigcupdot\mathcal{R}=\pi[\Sigma^\#]$.
%In total
%$$\widehat{\pi}[\Sig^\circ]\subseteq\widehat{\pi}[\Sig^\#]\subseteq\pi[\Sigma^\#]\subseteq\widehat{\pi}[\Sig^\circ].$$
%So all sets are equal.
\end{proof}
\begin{cor}\label{cafeashter}
%For the coding described in Theorem \ref{BMS}, 
Let $p$ be a $\chi$-hyperbolic periodic point, such that $\HWT_\chi^{\epsista_\chi}\neq\varnothing$. Then $\forall \epsilon\in (0,\epsista_\chi]$, $$\widehat{\pi}[\Sig^\#_\epsilon]=\pi[\Sistasharpepsilon]=\HWT_\chi^\epsilon.$$
\end{cor}
\begin{proof}
In Theorem \ref{BMS} we saw $\pi[\Sistasharpepsilon]= \HWT_\chi^\epsilon$. %For each chain  $\ul{u}\in\Sigma^\#$, %by \cite[lemma~4.5]{SBO}, $\pi(\ul{u})\in \chi-$hyp. Write $u_i=\psi_{x_i}^{p^s_i,p^u_i}$ and $\pi(\ul{u}):=x$, then $q_\epsilon(f^i(x)):=\min\{p^s_i,p^u_i\}$ satisfies the assumption so that $x\in NUH_\chi^\#$. Therefore $\pi[\Sigma^\#]\subseteq NUH_\chi^\#$.
In Proposition \ref{imagecanonic} we showed the quick argument for the equality $\widehat{\pi}[\Sig_\epsilon^\#]=\pi[\Sistasharpepsilon]$% (it follows from the local finiteness of $\mathcal{R}_\epsilon$ as a refinement of $\mathcal{Z}_\epsilon$, i.e. items (2)(a),(2)(b) in Definition \ref{Doomsday})
. Therefore we are done.
\end{proof}
\begin{definition}
	$\HWT_\chi^\star:=\HWT_\chi^{\epsista_\chi}$ is called the {\em recurrently-codable} points.
\end{definition}
Notice, by Claim \ref{canonicalepsilon}, $\bigcup\limits_{0<\epsilon\leq\frac{2}{3}\epsista_\chi}\HWT_\chi^\epsilon\subseteq \HWT_\chi^\star$.
\section{Ergodic Homoclinic Classes and Maximal Irreducible Components}
In this section $\epsilon$ is fixed and equals $\epsista_\chi$. The $\epsilon$ subscript of $\Sig_\epsilon,\Sig_\epsilon^\#,\mathcal{R}_\epsilon,\Sigma_\epsilon,\Sigma_\epsilon^\#$ will be omitted to ease notation.

Let $p$ be a periodic point in $\chi\mathrm{-summ}$. Since $p$ is periodic, $\|C^{-1}_\chi(\cdot)\|$ is bounded along the orbit of $p$, and therefore $p\in \HWT_\chi^\star$. Every point $x\in \HWT_\chi^\star$ is (recurrently-)codable, and so has a local stable manifold $V^s(x)$ (e.g. $V^s(\ul{u})$, $\ul{u}\in\pi^{-1}[\{x\}]\cap \Sigma^\#$), and a global stable manifold $W^s(x):=\bigcup_{n\geq0}f^{-n}[V^s(f^n(x))]$ (similarly for a global unstable manifold).
\begin{definition}\label{homoclinicclass}The {\em ergodic homoclinic class} of $p$ is
    $$H(p):=\left\{x\in \HWT_\chi^\star:W^u(x)\pitchfork W^s(o(p))\neq\varnothing,W^s(x)\pitchfork  W^u(o(p))\neq\varnothing\right\},$$
    where $\pitchfork$ denotes transverse intersections of full codimension, $o(p)$ is the (finite) orbit of $p$, and $W^{s}(\cdot),W^{u}(\cdot)$ are the global stable and unstable manifolds of a point (or points in an orbit), respectively.
    \end{definition}
\noindent This notion was introduced in \cite{RodriguezHertz}, with a set of Lyapunov regular points replacing $\HWT_\chi^\star$. Every ergodic conservative $\chi$-hyperbolic measure, is carried by an ergodic homoclinic class of some periodic hyperbolic point. 
\begin{definition}\label{irreducibility}
%Let $\mathcal{V}$ be the set of vertices of the graph $\mathcal{G}$ which induces $\Sigma$, the TMS which codes $NUH_\chi^\#$ that is constructed in \cite{SBO}.
%Consider the Markov partition $\mathcal{R}$ mentioned in \cite[\textsection~6,proof of Theorem 1.1]{SBO} for the multidimensional case (as done for the two-dimensional case in \cite[Definition~11.1]{Sarig}).
\begin{enumerate}
    \item Define $\sim\subseteq\mathcal{R}\times\mathcal{R}$ by $R\sim S\iff \exists n_{RS},n_{SR}\in\mathbb{N}\text{ s.t. } R\xrightarrow[]{n_{RS}}S,S\xrightarrow[]{n_{SR}}R$, i.e. there is a path of length $n_{RS}$ connecting $R$ to $S$, and a path of length $n_{SR}$ connecting $S$ to $R$. The relation $\sim$ is transitive and symmetric. When restricted to $\{R\in \mathcal{R}:R\sim R\}$, it is also reflexive, and thus an equivalence relation. Denote the corresponding equivalence class of some representative $R\in\mathcal{R}$, $R\sim R$ by $\langle R\rangle$.
    \item A {\em maximal irreducible component} in $\Sig$, corresponding to $R\in\mathcal{R}$ s.t. $R\sim R$, is $\{\ul{R}\in\Sig: \ul{R}\in\langle R\rangle^\mathbb{Z}\}$.
\end{enumerate}
\end{definition}

\begin{lemma}\label{avecOmri}
Let $p\in \chi\mathrm{-summ}$ s.t. $\exists l\in\mathbb{N}$ s.t. $f^l(p)=p$, then $p\in\chi\mathrm{-hyp}$.
\end{lemma}
\begin{proof}
We prove an exponential contraction strictly stronger than $e^{-\chi}$ on $H^s(p)$. The case for $H^u(p)$ is similar. First assume that $f(p)=p$. Since $p$ is $\chi$-summable, $\forall\xi\in H^s(p)$ with $|\xi|=1$, $\sum_{m=0}^\infty|d_pf^m\xi|^2e^{2\chi m}<\infty$. Let $\{\xi_i\}_{i=1}^{s(p)}$ be an orthonormal basis for $H^s(p)$ (w.r.t $\langle \cdot,\cdot\rangle_p$, the Riemannian form at $T_pM$). For all $m\geq0$, $\exists \xi^{(m)}\in H^s(p)$ with $|\xi^{(m)}|=1$ s.t. $|d_pf^m\xi^{(m)}|=\|d_pf^m|_{H^s(p)}\|$. Whence, for $a_i^{(m)}:=\langle \xi^{(m)},\xi_i\rangle_p$,
\begin{align}\label{jordan}
    \sum_{m=0}^\infty\|d_pf ^m |_{H^s(p)}\|^2e^{2\chi m}=&\sum_{m=0}^\infty|d_pf^m\xi^{(m)}|^2e^{2\chi m}\leq \sum_{m=0}^\infty\left(\sum_{i=1}^{s(p)}|a_i^{(m)}|\cdot|d_pf^m\xi_i|\right)^2e^{2\chi m}\\
    \leq&\sum_{m=0}^\infty\left(\sum_{i=1}^{s(p)}|d_pf^m\xi_i|\right)^2e^{2\chi m}=\sum_{i,j=1}^{s(p)}\sum_{m=0}^\infty\left(|d_pf^m\xi_i|e^{\chi m}\right)\left(|d_pf^m\xi_j|e^{\chi m}\right) \nonumber\\
    \leq &\sum_{i,j=1}^{s(p)}\sqrt{\sum_{m=0}^\infty|d_pf^m\xi_i|^2e^{2\chi m}}\cdot\sqrt{\sum_{m=0}^\infty|d_pf^m\xi_j|^2e^{2\chi m}}\nonumber\\
    \leq& d^2\max_{i\leq s(p)}\left\{\sum_{m=0}^\infty|d_pf^m\xi_i|^2e^{2\chi m}\right\}<\infty, \nonumber
\end{align}
where the third line is by the Cauchy-Schwarz inequality, and the second line is since $|a_i^{(m)}|\leq1,i\leq s(p)$ (by the Cauchy-Schwarz inequality as well).
$d_pf:H^s(p)\rightarrow H^s(p)$ is a linear map, then by working in coordinates, it is sufficient to assume w.l.o.g. that $d_pf$ is of the form of a Jordan block $J_{s(p)}(\lambda)$ for some $\lambda\in \mathbb{R}\setminus\{0\}$ (since $f$ is a diffeomorphism).
%\[ A:=\begin{bmatrix}
%    & \lambda & 1 & 0 &  \dots & 0 &\\
%     &     & \lambda & 1 & 0 &  &\\
%        &   &  & \ddots & \ddots  &  &\\
%        &   &  &  &   & 1 &\\
%        &    &  &  &  & \lambda&
%\end{bmatrix},\text{ for some }\lambda\in\mathbb{R}.
%\]
Let $N$ denote the nilpotent matrix whose superdiagonal (entries right above the diagonal) is all ones, and all other entries are zero, and $I$ denotes the identity matrix. Then $N^{s(p)}=0$, and trivially $N$ and $\lambda I$ commute. Then $J_{s(p)}(\lambda)=\lambda I+N$, and by the binomial theorem (w.l.o.g. $m\geq s(p)$),
\begin{equation}\label{jordan2}
J_{s(p)}(\lambda)^m=\left(\lambda I+N\right)^m=\sum_{k=0}^{s(p)}\binom{m}{k}\lambda^{m-k}N^k.
\end{equation}
Since $J_{s(p)}(\lambda)$ is a Jordan block, it admits an eigenvector with an eigenvalue $\lambda$, and so $\|J_{s(p)}(\lambda)\|\geq |\lambda|$, and 
%Therefore, $\exists C=C(d,p)>0$ s.t. $\forall m\geq s(p)$, 
$\|d_pf^m|_{H^s(p)}\|\geq |\lambda|^m$. Hence, by equation \eqref{jordan}, $\sum_{m=0}^\infty|\lambda|^{2m}e^{2\chi m}\leq \sum_{m=0}^\infty\|d_pf^m|_{H^s(p)}\|^2e^{2\chi m}<\infty$, and so $0<|\lambda|<e^{-\chi}$. On the other hand, by equation \eqref{jordan2}, $\|d_pf^m|_{H^s(p)}\|\leq C\cdot m^{s(p)}\cdot|\lambda|^m$ for some $C=C(p)>0$. Whence, $\|d_pf^m|_{H^s(p)}\|\leq C \cdot m^{s(p)}\cdot e^{-\chi' m}$, where $\chi':=\log |\lambda|<-\chi$. Thus, $\limsup_{m\rightarrow\infty}\frac{1}{m}\log\|d_pf^m|_{H^s(p)}\|<-\chi$. This concludes the proof for the case $f(p)=p$.

In the case where the period of $p$ is $l>1$, write $\forall \xi\in H^s(p)$,
\begin{equation*}
    \sum_{m=0}^\infty|d_pf^m\xi|^2e^{2\chi m}=\sum_{i=0}^{l-1}e^{2i\chi}\sum_{m=0}^{\infty}|d_{f^i(p)}(f^l)^{m}(d_pf^i\xi)|^2e^{2(l\chi)m}<\infty.
\end{equation*}
Then by the first part of this proof, $\exists \lambda_i$, $i=0,...,l-1$ s.t. $0<|\lambda|<e^{-\chi\cdot l}$ and $C_i=C_i(p)>0$ s.t. $\forall 0\leq i\leq l-1, l\geq0$, $\|d_pf^{m\cdot l+i}|_{H^s(p)}\|\leq \max_{i}\{C_i\}\max_i\{|\lambda_i|^{m}\}\cdot m^{s(p)}$. As in the case of $f(p)=p$, this is sufficient. 
\end{proof}
 
The following two definitions are due to Sarig in \cite[\textsection~4.2,Definition~4.8]{Sarig} (the version here corresponds to the case $d\geq2$ from \cite[Definition~3.1,Definition~3.2]{SBO}).
\begin{definition}\label{def135} Let $x\in \HWT_\chi^\star$, a {\em $u-$manifold} in $\psi_x$ is a manifold $V^u\subset M$ of the form
$$V^u=\psi_x[\{(F_1^u(t_{s(x)+1},...,t_d),...,F_{s(x)}^u(t_{s(x)+1},...,t_d),t_{s(x)+1},...t_d) : |t_i|\leq q\}],$$
where $0<q\leq Q_\epsilon(x)$, and $\overrightarrow{F}^u$ is a $C^{1+\beta/3}$ function s.t. $\max\limits_{\overline{R_q(0)}}|\overrightarrow{F}^u|_\infty\leq Q_\epsilon(x)$.

Similarly we define an {\em $s-$manifold} in $\psi_x$:
$$V^s=\psi_x[\{(t_1,...,t_{s(x)},F_{s(x)+1}^s(t_1,...,t_{s(x)}),...,F_d^s(t_1,...,t_{s(x)})): |t_i|\leq q\}],$$
with the same requirements for $\overrightarrow{F}^s$ and $q$. We will use the superscript ``$u/s$" in statements which apply to both the $u$ case and the $s$ case. The function $\overrightarrow{F}=\overrightarrow{F}^{u/s}$ is called the {\em representing function} of $V^{u/s}$ at $\psi_x$. The parameters of a $u/s$ manifold in $\psi_x$ are: 
%(in this context for a matrix $A$: 
\begin{itemize}
\item $\sigma-$parameter: $\sigma(V^{u/s}):=\|d_{\cdot}\overrightarrow{F}\|_{\beta/3}:=\max\limits_{\overline{R_q(0)}}\|d_{\cdot}\overrightarrow{F}\|+\text{H\"ol}_{\beta/3}(d_{\cdot}\overrightarrow{F})$,

where $\text{H\"ol}_{\beta/3}(d_{\cdot}\overrightarrow{F}):=\max\limits_{\vec{t_1},\vec{t_2}\in\overline{R_q(0)}}\{\frac{\|d_{\overrightarrow{t_1}}\overrightarrow{F}-d_{\overrightarrow{t_2}}\overrightarrow{F}\|}{|\overrightarrow{t_1}-\overrightarrow{t_2}|^{\beta/3}}\}$ and $\|A\|:=\sup\limits_{v\neq0}\frac{|Av|_\infty}{|v|_\infty}$.
\item $\gamma-$parameter: $\gamma(V^{u/s}):=\|d_0\overrightarrow{F}\|$
\item $\varphi-$parameter: $\varphi(V^{u/s}):=|\overrightarrow{F}(0)|_\infty$
\item $q-$parameter: $q(V^{u/s}):=q$
\end{itemize}

A {\em $(u/s,\sigma,\gamma,\varphi,q)-$manifold} in $\psi_x$ is a $u/s$ manifold $V^{u/s}$ in $\psi_x$ whose parameters satisfy $\sigma(V^{u/s})\leq\sigma,\gamma(V^{u/s})\leq\gamma,\varphi(V^{u/s})\leq\varphi,q(V^{u/s})\leq q$.
\end{definition}
Notice that the dimensions of an $s$ or a $u$ manifold in $\psi_x$ depend on $x$. Their sum is $d$.
\begin{definition}\label{admissible} Suppose $x\in \HWT_\chi^
\star$ and $0<p^s,p^u\leq Q_\epsilon(x)$ (i.e. $\psi_x^{p^s,p^u}$ is a double Pesin-chart)%, see \cite[\textsection~2.3.1]{SBO})
. A $u/s$-admissible manifold in $\psi_x^{p^s,p^u}$ is a {\em $(u/s,\sigma,\gamma,\varphi,q)-$manifold} in $\psi_x$ s.t.
$$\sigma\leq\frac{1}{2},\gamma\leq\frac{1}{2}(p^u\wedge p^s)^{\beta/3},\varphi\leq10^{-3}(p^u\wedge p^s),\text{ and }q = \begin{cases} p^u & u-\text{manifolds} \\ 
p^s & s-\text{manifolds} \end{cases}.$$
\end{definition} Recall: $\forall \ul{u}\in \Sigma$ there exists a local stable manifold for $\pi(\ul{u})$, $V^s(\ul{u})= V^s((u_i)_{i\geq0})$, and a local unstable manifold for $\pi(\ul{u})$, $V^u(\ul{u})= V^u((u_i)_{i\leq0})$ (see \cite[Proposition~3.12,Proposition~4.4]{SBO}; or \cite[Proposition~4.15,Proposition~6.3]{Sarig} in the case $d=2$). $V^s(\ul{u})$, $V^u(\ul{u})$ are admissible manifolds in $u_0=\psi_{x_0}^{p^s_0,p^u_0}$.%, in the sense that they are the graphs of $C^1$-smooth functions in $\psi_{x_0}^{\min\{p^s_0,p_0^u\}}$ (see \cite[Definition~3.2]{SBO}); and a Pesin-chart $\psi_{x}^{\alpha}$ ($x\in \HWT_\chi^\star,0<\alpha\leq Q_\epsilon(x)$) is the map $\exp_x\circ C_\chi:\{v\in\mathbb{R}^d: \max_{i\leq d}|v_i|\leq \alpha\}\rightarrow B_\epsilon(x)$; and a double Pesin-chart $\psi_{x}^{p^s,p^u}$ is the ordered couple $(\psi_x^{p^s}, \psi_x^{p^u})$ (see \cite[\textsection~2.3.1]{SBO}).

\begin{definition}\label{Ledrappier} (This definition was introduced in \cite[Lemma~4.6]{Sarig}, and is due to F. Ledrappier) $\forall x\in \HWT_\chi^\star$, %if $\alpha>0$ is a constant number such that $\forall n\in\mathbb{Z}$, $Q_\epsilon(f^n(x))>\alpha$, then 
$$p^u_n(x):=\max\{t\in\{e^{-\frac{\ell\epsilon}{3}}\}_{\ell\geq0}:e^{-\epsilon N}t\leq Q_\epsilon(f^{n-N}(x)),\forall N\geq0\},$$ 
$$p^s_n(x):=\max\{t\in\{e^{-\frac{\ell\epsilon}{3}}\}_{\ell\geq0}:e^{-\epsilon N}t\leq Q_\epsilon(f^{n+N}(x)),\forall N\geq0\}.$$
\end{definition} 
Note: the chain $\{\psi_{f^n(x)}^{p^s_n(x),p^u_n(x)}\}_{n\in\mathbb{Z}}$ is admissible
 (see \cite[Definifion~2.23]{SBO} for the conditions for an edge between two double Pesin-charts).
\begin{lemma}\label{ForUniformHyperbolicity}
	Let $p$ be a $\chi$-hyperbolic periodic point, and let $\ul{u}$ be the admissible (periodic) chain
	 $\{\psi_{f^n(p)}^{p^s_n(p),p^u_n(p)}\}_{n\in\mathbb{Z}}$. Let $x\in V^u(\ul{u})\cap \HWT_\chi^\star$ s.t. $\mathrm{dim}H^s(x)= \mathrm{dim}H^s(p)$. Then $$\limsup_{n\rightarrow\infty} \sup_{\xi_n\in H^s(f^{-n}(x)),|\xi_n|=1}S(f^{-n}(x),\xi_n)\leq \max_i\{\sup_{\eta_i\in H^s(f^i(p)),|\eta_i|=1}S(f^i(p),\eta_i)\}<\infty.$$ A similar claim holds for $x\in V^s(\ul{u})\cap \HWT_\chi^\star$.
\end{lemma}
\begin{proof}
\begin{comment}	
Let $\delta>0$, and assume w.l.o.g. that $f(p)=p$. By the Inclination Lemma (\cite[thm.~5.7.2]{BrinStuck}), and by the $\epsilon$-weak temperability of $x$, we may assume w.l.o.g. that $\forall i\geq0$ $d_{C^1}(V^s(f^{-i}(x)), V^s(\ul{u}))\leq\delta Q_\epsilon(p)$, where $V^s(f^{-i}(x))$ is the part of $W^s(f^{-i}(x))$ which is close in $C^1$-norm to $V^s(\ul{u})$, and the distance is calculated in the chart $\psi_p^{Q_\epsilon(p)}$. In particular, since $V^s(\ul{u})$ is an admissible manifold in $\psi_p^\tau$, $V^s(f^{-i}(x))$ is the graph of a $C^1$-smooth function. This allows us to use \cite[Lemma~4.6]{SBO}: $\exists \pi^{(n)}_p:H^s(f^{-n}(x))\rightarrow H^s(p)$,$n\geq0$ a sequence of linear maps s.t. $\forall \xi\in H^s(x)\setminus\{0\}$, 
\begin{align*}
\rho_{n+1}(\xi)\geq& e^{\sqrt{\epsilon}}\Rightarrow \rho_n(\xi)\leq \rho_{n+1}(\xi)e^{-Q_\epsilon^\frac{\beta}{6}(p)}\text{ where ,}\\
\rho_n(\xi):=&\max\{\frac{S(f^{-n}(x),d_x(f^{-n})\xi)}{S(p,\pi_p^{(n)} d_x(f^{-n})\xi)}, \frac{S(p,\pi_p^{(n)} d_x(f^{-n})\xi)} {S(f^{-n}(x), d_x(f^{-n})\xi)}\},\forall n\geq0. 
\end{align*}

That is, the ratio of the $S(\cdot,\xi)$ functions can only improve (by a fixed improvement) when starting at far $f^{-n}(x)$ and going back to $x$ with iterations of $f$. But, $x\in NUH_\chi^\#$, and as such, in particular, $\liminf \sup_{\xi_n\in H^s(f^{-n}(x)),|\xi_n|=1}S(f^{-n}(x),\xi_n)<\infty$. Therefore, assume for contradiction that $\rho_0>e^{\sqrt{\epsilon}}$, and let $N:=\lceil\log_{e^{Q_\epsilon^{\beta/6}(p)}}\rho_0\rceil$; write 
\end{comment}

Fix $\delta\in(0,1)$. Assume w.l.o.g. that $f(p)=p$. By Definition \ref{Ledrappier}, if $f(p)=p$, then $\sigma\ul{u}=\ul{u}$. %Write $\tau:= Q_\epsilon(p)$. 
By the Inclination Lemma (\cite[Theorem~5.7.2]{BrinStuck})% and by the $\epsilon$-weak temperability of $x$
, we may assume w.l.o.g. that $d_{C^1}(V^s(f^{-i}(x)), V^s(\ul{u}))\leq\delta$ $\forall i\geq0$, where $V^s(f^{-i}(x))$ is the part of $W^s(f^{-i}(x))$ which is close in $C^1$-norm to $V^s(\ul{u})$, and the $C^1$-distance is calculated in the chart $\psi_p^{Q_\epsilon(p)}$. In particular, since $V^s(\ul{u})$ is an admissible manifold in $\psi_p^{Q_\epsilon(p)}$, $V^s(f^{-i}(x))$ is the graph of a $C^1$-smooth function. Denote the function representing the graph of $V^s(\ul{u})$ by $F$, and the function representing the graph of $V^s(f^{-i}(x))$ by $G_i$, $i\geq0$. Let $P_s:\mathbb{R}^{d}\rightarrow \mathbb{R}^{s(x)}$ be the projection to the $s(x)$ first coordinates, and let $\xi\in H^s(x)$ s.t. $|d_x(f^{-1})\xi|=1$. $\xi=d_{\psi_p^{-1}(x)}\psi_p(u,d_{P_s\psi_p^{-1}(x)}G_0u)$ for some $u\in \mathbb{R}^{s(p)}$. Define $\eta=\eta(\xi):= d_0\psi_p(u,d_0Fu)\in H^s(p)$. Write again $d_x(f^{-1})\xi= d_{\psi_p^{-1}(f^{-1}(x))}\psi_p(v,d_{P_s\psi_p^{-1}(f^{-1}(x))}G_1v) $ for some $v\in\mathbb{R}^{s(x)}$, and define $\zeta=\zeta(\xi):= d_0\psi_p(v,d_0Fv)\in H^s(p)$. Notice, $\eta:T_xV^s(x)\rightarrow H^s(p)$ is a linear map, and in fact the definition extends natutally to $\eta:T_{f^{-i}(x)}V^s(f^{-i}(x))\rightarrow H^s(p)$, $\forall i\geq0$. Thus, by the Inclination Lemma, assume w.l.o.g. that $\Big|| d_x(f^{-1})\xi|^2-|d_p(f^{-1})\eta|^2\Big|\leq \delta$. Define $\rho:=\max\{\frac{S(x,\xi)}{S(p,\eta)},\frac{S(p,\eta)}{S(x,\xi))}\}$. 

Step 1: \begin{align*}
 		S^2(f^{-1}(x),d_x(f^{-1})\xi)=&2\sum_{m=0}^\infty|d_{f^{-1}(x)}f^md_x(f^{-1})\xi|^2e^{2\chi m}=S^2(x,\xi)e^{2\chi}+2|d_x(f^{-1})\xi|^2\\
 		\leq& \rho^2e^{2\chi}S^2(p,\eta)+2|d_x(f^{-1})\xi|^2.\\
 		S^2(p,d_p(f^{-1})\eta)=&2\sum_{m=0}^\infty|d_{p}f^md_p(f^{-1})\eta|^2e^{2\chi m}=S^2(p,\eta)e^{2\chi}+2|d_p(f^{-1})\eta|^2.
\end{align*}
Then,
\begin{align}\label{fridaynoon}
	\frac{S^2(f^{-1}(x),d_x(f^{-1})\xi)}{S^2(p,d_p(f^{-1})\eta)}\leq&\frac{\rho^2e^{2\chi}S^2(p,\eta)+2|d_x (f^{-1})\xi|^2}{S^2(p,\eta)2^{2\chi}+2|d_p (f^{-1})\eta|^2}\\
	=&\rho^2-\frac{2(\rho^2-1)|d_p(f^{-1})\eta|^2+2(|d_p (f^{-1}) \eta|^2-|d_x (f^{-1}) \xi|^2)}{S^2(p,d_p(f^{-1})\eta)}.\nonumber
\end{align}	
Now, if $\rho\geq e^{\sqrt{\delta}}$, then $\rho^2-1\geq 2\sqrt{\delta}$, and so $(\rho^2-1)|d_p(f^{-1})\eta|^2+(|d_p (f^{-1}) \eta|^2-|d_x (f^{-1}) \xi|^2)\geq (\rho^2-1)(1-\delta)-\delta\geq (\rho^2-1)(1-\delta)-(\rho^2-1)\frac{\delta}{2\sqrt{\delta}}=(\rho^2-1)(1-\delta-\frac{\sqrt{\delta}}{2})\geq(\rho^2-1)(1-\frac{3}{2}\sqrt{\delta})\geq(\rho^2-1)e^{-2\sqrt{\delta}}$, for small enough $\delta\in(0,1)$. We then get that all together,
\begin{align*}
	\frac{S^2(f^{-1}(x),d_x(f^{-1})\xi)}{S^2(p,d_p(f^{-1})\eta)}\leq&\rho^2-\frac{2(\rho^2-1)e^{-2\sqrt{\delta}}}{S^2(p,d_p(f^{-1})\eta)}\leq \rho^2-\frac{2(\rho^2-1)e^{-2\sqrt{\delta}}}{\|C_\chi^{-1}(p)\|\cdot |d_p(f^{-1})\eta|}\leq \rho^2-\frac{2(\rho^2-1)e^{-2\sqrt{\delta}}}{\|C_\chi^{-1}(p)\|(1+\delta)}\\
	\leq&\rho^2\left(1-\frac{2(1-\frac{1}{\rho^2})e^{-2\sqrt{\delta}}}{\|C_\chi^{-1}(p)\|e^{\delta}}\right)\leq \rho^2\left(1-\frac{2(1-\frac{1}{e^{2\sqrt{\delta}}})e^{-2\sqrt{\delta}}}{\|C_\chi^{-1}(p)\|e^{\delta}}\right)\leq \rho^2\left(1-\frac{2\sqrt{\delta}e^{-2\sqrt{\delta}}}{\|C_\chi^{-1}(p)\|e^{\delta}}\right). 
\end{align*}
For $\delta\in(0,1)$ small enough so $\frac{\sqrt{\delta}^{-\frac{1}{2}}e^{-1}}{\|C_\chi^{-1}(p)\|}\geq 1$, %Since $\tau\leq Q_\epsilon(p)\leq \frac{1}{4\|C_\chi^{-1}(p)\|^2}$, we get
$$\frac{S^2(f^{-1}(x),d_x(f^{-1})\xi)}{S^2(p,d_p(f^{-1})\eta)}\leq \rho^2(1-2\delta)\leq\rho^2e^{-2\delta}.$$
Since both $\zeta$ and $\eta$ depend continuously on $\xi$, and can be made arbitrarily close with $\delta>0$ small enough, and since $S^2(p,\cdot)$ is continuous (see \cite[Theorem~2.8]{SBO}), we may assume w.l.o.g. that $\frac{S^2(p,\zeta)}{S^2(p,d_p(f^{-1})\eta)}\in [e^{-\delta}, e^{\delta}]$. Thus we get in total,
\begin{equation}\label{laurasleeps}
\frac{S^2(f^{-1}(x),d_x(f^{-1})\xi)}{S^2(p,\zeta)}\leq \rho^2e^{-\delta}.	
\end{equation}
Similarly one obtains $\frac{S^2(f^{-1}(x),d_x(f^{-1})\xi)}{S^2(p,\zeta)}\geq \rho^{-2}e^{\delta}$. 

\medskip
Step 2: In step 1, we fixed $\delta>0$, and made some assumptions without losing generality, that hold for all $f^{-n}(x)$ for $n\geq0$ large enough, and hence don't affect the limit. Denote by $n_\delta\geq0$ the smallest backward-iteration of $x$ to satisfy this way the choice of $\delta$ (thus step 1 treats all $f^{-n}(x)$ $\forall n\geq n_\delta$). That means that the ratio in equation \eqref{laurasleeps} is either in $[e^{-\sqrt{\delta}}, e^{\sqrt{\delta}}]$, or it improves by a factor of at least $e^{\delta}$, with each iteration of $f^{-1}$ (starting from $f^{-n_\delta}(x)$). On the other hand, when $1\leq\rho\leq e^{\sqrt{\delta}}$, by equation \eqref{fridaynoon}, the bound may deteriorate by a factor of at most $e^{2\delta}$. Thus,
$$\limsup_{n\rightarrow\infty} \frac{S^2(f^{-n}(x),d_x(f^{-n})\xi)}{S^2(p,\eta(d_x(f^{-n})\xi))}\leq e^{\sqrt{\delta}+2\delta}, \liminf_{n\rightarrow\infty} \frac{S^2(f^{-n}(x),d_x(f^{-n})\xi)}{S^2(p,\eta(d_x(f^{-n})\xi))}\geq e^{-\sqrt{\delta}-2\delta}.$$
Now, since $\delta>0$ was arbitrary, we get $\forall\xi\in H^s(x)\setminus\{0\}$,
\begin{equation}\label{snirsleeps}
\lim_{n\rightarrow\infty} \frac{S^2(f^{-n}(x),d_x(f^{-n})\xi)}{S^2(p,\eta\left(d_x(f^{-n})\xi\right))}=1\text{, and 
so } \lim_{n\rightarrow\infty} \frac{S^2\left(f^{-n}(x),\frac{d_x(f^{-n})\xi}{|d_x(f^{-n})\xi|}\right)}{S^2\left(p,\eta(\frac{d_x(f^{-n})\xi}{|d_x(f^{-n})\xi|}) \right)}=1.	
\end{equation}
From that, since $\lim\limits_{n\rightarrow\infty}|\eta(\frac{d_x(f^{-n})\xi)}{|d_x(f^{-n})\xi)|})|=1$ (by the Inclination Lemma and the definition of $\eta:T_{f^{-n}(x)}V^s(f^{-n}(x))\rightarrow H^s(p)$), the lemma follows.
\end{proof}
The reason that we got a better result here, than \cite[Lemma~4.6]{SBO} and \cite[Lemma~7.2]{Sarig}, is that here the centers of charts are $f^{-i}(p)$, and there is no distortion as a consequence of non-full overlap between $f^{-1}(x_i)$ and $x_{i-1}$.
\begin{definition}
	$\forall x\in \HWT_\chi^\star=\bigcupdot\mathcal{R}$, {\em the itinerary of $x$} is $\ul{R}(x):=(R(f^i(x)))_{i\in\mathbb{Z}}$, where $R(f^i(x)):=$unique element of $\mathcal{R}$ which contains $f^i(x)$.
\end{definition}
Notice that $\ul{R}(\cdot):\HWT_\chi^\star\rightarrow \Sig$ is a bijection onto its image and is a measurable map s.t. $\ul{R}\circ f=\sigma\circ \ul{R}$ and $\widehat{\pi}\circ \ul{R}=\mathrm{Id}$ (in particular $\forall x\in \HWT_\chi^\star$, $\ul{R}(x)$ is an admissible chain in $\Sig^\#$, as demonstrated in the proof of Proposition \ref{imagecanonic}).
\begin{prop}\label{irreduciblecoding}
For every $\chi$-hyperbolic periodic point $p$, $\exists\tilde{\Sigma}\subseteq\Sig$ a maximal irreducible component, s.t. $\widehat{\pi}[\tilde{\Sigma}]\supseteq H(p)$ modulo all conservative (possibly infinite) measures which are carried by $H(p)$.
\end{prop}
This is an adaptation of the proof by Buzzi, Crovisier and Sarig in \cite[Lemma~3.11,Lemma~3.12]{BCS}.
%Proposition \ref{irreduciblecoding} (which appears and is proven in the appendix) states that for every periodic hyperbolic point $p$, there exists a maximal irreducible component $\tilde{\Sigma}\subseteq\Sig$ s.t. $\widehat{\pi}[\tilde{\Sigma}]\supseteq H(p)$. The following proposition is a slightly additional analysis of the proof of its proof.
\begin{proof} Let $\mu$ be a conservative measure carried by $H(p)$. By Corollary \ref{cafeashter} and Proposition \ref{imagecanonic} $H(p)\subseteq \widehat{\pi}[
\Sig^\#]=\bigcupdot\mathcal{R}$, and so $\widehat{\mu}:=\mu\circ \underline{R}^{-1}$ is a well defined invariant and conservative measure on $\Sig^\#$, and $\widehat{\mu}\circ\widehat{\pi}^{-1}=\mu$. Whence $\widehat{\mu}$ is carried by 
$$\Sig^{\#\#}:=\left\{\ul{R}\in\Sig^\#:\exists a\in\mathcal{R}\text{ s.t. }R_i=a\text{ for infinitely many }i\geq0\text{ and for infinitely many }i\leq0\right\}.$$
Therefore, $\widehat{\pi}[\Sig^{\#\#}]\cap H(p)$ carries all conservative measures which are carried by $H(p)$. Each chain in $\Sig^{\#\#}$ has the following form:
\begin{equation}\label{periodicform}
    ...a,\tilde{p}_{-i},a,\tilde{p}_{-i+1},a...a,\tilde{p}_{-1},a...a,\tilde{p}_1,a...a,\tilde{p}_{i-1},a...a,\tilde{p}_i,a...
\end{equation}
where $\tilde{p}_{i'}$ denotes a finite word which connects $a$ to $a$. For each word $w_{n}:=a,\tilde{p}_{-n},a, \tilde{p}_{-n+1},a,...,a, \tilde{p}_{n-1},a, \tilde{p}_{n}a$, let $x_{\ul{R}}^{(n)}$ be the image of the periodic extension of $w_n$.

\textbf{Step 0:} For every chain $\ul{R}$ as in equation \eqref{periodicform}, let $\ul{R}^{(n)}:=$the admissible concatenation of $w_n$ to itself%...a,\tilde{p}_{n},a,\tilde{p}_{n},a,\tilde{p}_{-n},a,\tilde{p}_{-n+1},a...a,\tilde{p}_{-1},a...a,\tilde{p}_1,a...a,\tilde{p}_{i-1},a...a,\tilde{p}_n,a,\tilde{p}_n,a,\tilde{p}_n...\text{ }$
. Then $\ul{R}^{(n)}\rightarrow\ul{R}$. As demonstrated in the proof of Proposition \ref{imagecanonic}, $\exists \ul{u}^{(n)},\ul{u}\in \Sigma^\#$ s.t. $\ul{u}^{(n)}\rightarrow \ul{u}$ and $\pi(\ul{u}^{(n)})=\widehat{\pi}(\ul{R}^{(n)}), \pi(\ul{u})=\widehat{\pi}(\ul{R})$. It follows that $V^u(\ul{u}^{(n)})\rightarrow V^u(\ul{u}),V^s(\ul{u}^{(n)})\rightarrow V^s(\ul{u})$ as admissible manifolds in $u_0$ (i.e. the representing functions converge in $\|\cdot\|_\infty$-norm). Since $\pi(\ul{u})=\widehat{\pi}(\ul{R})\in H(p)$, $\exists N=N_{\ul{R}}$ s.t. $f^N[V^u(\sigma^{-N}\ul{u})]\pitchfork W^s(o(p))\neq\varnothing$, $f^{-N}[V^s(\sigma^{N}\ul{u})]\pitchfork W^s(o(p))\neq\varnothing$. Whence, $\exists n_{\ul{R}}$ s.t. $\forall n\geq n_{\ul{R}}$, $f^N[V^u(\sigma^{-N}\ul{u}^{(n)})]\pitchfork W^s(o(p))\neq\varnothing$, $f^{-N}[V^s(\sigma^{N}\ul{u}^{(n)})]\pitchfork W^s(o(p))\neq\varnothing$. Let $P_{\ul{R}}:=\{x_{\ul{R}}^{(n)}\}_{|n|>n_{\ul{R}}}$. Then $\forall n\geq n_{\ul{R}}$, $\widehat{\pi}(\ul{R}^{(n)})\in H(p)$, and $W^s(\widehat{\pi}(\ul{R}^{(n)}))=W^s(x_{\ul{R}}^{(n)})$, $W^u(\widehat{\pi}(\ul{R}^{(n)}))=W^u(x_{\ul{R}}^{(n)})$; therefore $P_{\ul{R}}\subseteq H(p)$.

Consider the countable collection of all periodic points generated in this manner $\{p_i\}_{i\geq0}=\bigcup\{P_{\ul{R}}:\ul{R}\in \Sig^{\#\#}\cap\widehat{\pi}^{-1}[\{x\}],x\in H(p)\}$ ($\subseteq \HWT_\chi^\star$, Theorem \ref{BMS}). Then by the transitivity of the homoclinic relation (\cite[Proposition~2.1]{Newhouse}), $\forall i,j\geq0$, $p_i\in H(p_j)$. Assume w.l.o.g. that $\exists N_l\uparrow\infty$ s.t. $\forall l\in \mathbb{N},\forall i\leq N_l$, $o(p_i)\subseteq\{p_j\}_{j\leq N_l}$. Fix $N\in\{N_l\}_{l\geq0}$.

\textbf{Step 1:} $\forall 0\leq i,j\leq N$ $\exists t_{ij}\in \Big(W^u(p_i)\pitchfork W^s(p_j)\Big)\cap H(p)$, and $t_{ij}$ has a uniformly hyperbolic orbit, and its coding involves finitely many symbols.

\textbf{Proof:} %Every periodic point $p_i$, for $i\leq N$, is generated as a sequence in the coding of a point $x\in H(p)$. That is, $p_i\in W^u(f^{n_u}(x))\pitchfork W^s(f^{n_s}(x))$ for some $n_u,n_s\in\mathbb{Z}$, where in particular $x\in \chi\mathrm{-hyp}$. Therefore, $p_i\in \chi\mathrm{-hyp}$ (for a step-by-step proof, see \cite[lemma~4.5]{SBO}, w.r.t. to some other $\chi'\in(\chi,\chi(x))$ replacing $\chi$ in the proof); and since $p_i$ is periodic, $p_i\in NUH_\chi^\#$.
Since $p_i\in H(p_j)$, $\exists t_{ij}\in W^u(p_i)\pitchfork W^s(p_j)$. Showing that $t_{ij}$ has a uniformly hyperbolic orbit would yield that $t_{ij}\in \HWT_\chi^\star$, and so, since $p_i,p_j\in H(p)$, also $t_{ij}\in H(p)$.
%We are left to show the existence of $t_{ij}\in \Big(W^u(p_i)\pitchfork W^s(p_j)\Big)\cap H(p)$, for all $p_i,p_j\in H(p)$, $0\leq i,j\leq N$. By the fact that $p_i,p_j\in H(p)$, we get $W^u(o(p))\pitchfork W^s(p_j),W^u(o(p))\pitchfork W^s(p_i)\neq\varnothing$. Then by the Inclination Lemma again, similarly to the argument in the paragraph above which shows $p_i,p_j\in H(p)$, we get that $\exists t_{ij}\in W^u(p_i)\pitchfork W^s(p_j)$. Each such $t_{ij}$ is in $\chi\mathrm{-hyp}$ (again, this argument can be seen in full details in \cite[lemma~4.5]{SBO} with $\chi$ replaced by $\chi'\in(\chi,\min\{\chi(p_i),\chi(p_j)\})$ in the proof). In addition, the sequence of local unstable manifolds of $\{f^{k}(t_{ij})\}_{k\geq0}$ decomposes into $\mathrm{period}(p_j)$ disjoint subsequences s.t. each subsequence converges exponentially fast in $C^1$-norm to the local unstable manifold of some point in $o(p_j)$. The stable manifold of each $f^{k}(t_{ij})$ coincides with the stable manifold of $p_j$, whence the angle between $T_{f^k(t_{ij})}W^u(t_{ij})$ and $T_{f^k(t_{ij})}W^s(t_{ij})$ is bounded uniformly for all $k\geq0$. Similarly, with iterations of $f^{-1}$, the angle is bounded for $k\leq0$. In addition, 
By Lemma \ref{avecOmri}, $p_i,p_j\in \chi\mathrm{-hyp}$; whence, by \cite[Lemma~4.5]{SBO}, $t_{ij}\in \chi\mathrm{-hyp}$. By the Inclination Lemma (\cite[Theorem~5.7.2]{BrinStuck}), the angle between $W^s(t_{ij})$ and $W^u(t_{ij})$ is bounded away from zero along the orbit of $t_{ij}$. Therefore, by \cite[Lemma~4.5]{SBO} and by Lemma \ref{ForUniformHyperbolicity}, $\{\|C_\chi^{-1}(f^k(t_{ij}))\|\}_{k\in\mathbb{Z}}$ is bounded along the orbit of $t_{ij}$. Thus, $t_{ij}\in \HWT_\chi^\star$ and can be coded with finitely many symbols.\footnote{\label{finitelymanyletters}By \cite[Proposition~4.5]{Sarig}, there exists a $\ul{u}=\{\psi_{x_i}^{p^s_i,p^u_i}\}_{i\in\mathbb{Z}}\in\Sigma^\#$ s.t. $\pi(x_i)$ $\frac{Q_\epsilon(x_i)}{Q_\epsilon(f^i(x))}=e^{\pm\frac{\epsilon}{3}}$ $\forall i\in\mathbb{Z}$, and $p_i^u,p_i^s$ are given by a formula s.t. if $\inf_{n\in\mathbb{Z}}\{Q_\epsilon(x_n)\}>0$, then $\inf_{n\in\mathbb{Z}}\{p^s_n,p^u_n\}>0$; whence by the discreteness and local-finiteness of $\mathcal{Z}_\epsilon$ (Definition \ref{changesToCoding}(5), Definition \ref{Doomsday}(2)(b)), there are finitely many possible letters in that coding.}

\textbf{Step 2:} Let $\{t_{ij}\}_{i,j\leq N}$ be as defined in step 1, and choose some $\ul{\zeta}_{ij}\in \widehat{\pi}^{-1}[\{t_{ij}\}]$ (which involves only finitely many symbols). As defined before step 0, each point $p\in \{p_i\}_{i\geq0}$ is the image of a periodic extension of a finite word $w(p)=a,\widetilde{p},a$; in the following definition $\ul{w}(p_i)$ is the periodic extension of $w(p_i)	$ which induces $p_i$, $0\leq i\leq N$.
\begin{align*}
\tilde{L}:=\{&%\ul{u}=...p_{i_{-j}},p_{i_{-j}},p_{i_{-j}},p_{i_{-j+1}},...,p_{i_{-2}},p_{i_{-1}},W,p_{i_{1}},p_{i_{2}},...,p_{i_{j-1}},p_{i_{j}},p_{i_{j}}...\\
%&:\ul{u}\in\Sig,u_0\in W,|W|\leq N, \forall k, p_{i_k}\in\{p_l\}_{l=1}^N,j\leq N\}
\ul{w}(p_i):0\leq i\leq N\}\bigcup\{\ul{\zeta}_{ij}\}_{i,j\leq N}.
\end{align*}
%That is, besides $\{\ul{\zeta}_{ij}\}_{i,j\leq N}$, $\tilde{L}$ contains all admissible chains with some window of length $\leq N$, which outside of it involve only the $N$ first periodic patterns, and which stabilize on some periodic pattern at some point in the past and at some point in the future not further than $N$ periodic patterns of the window. In addition, all chains in $\tilde{L}$ have their center located in the window $W$. 
%Therefore, by the local-compactness of $\Sig$,
$\tilde{L}$ is finite and involves only finitely-many symbols.

\textbf{Step 3:}
$$\text{Define }L:=\bigcup_{\ul{u}\in\tilde{L}}\{\sigma^j\ul{u}\}_{j\in\mathbb{Z}}.$$
By the H\"older-continuity of $\widehat{\pi}$, it follows that $\forall y\in \widehat{\pi}[L]$, $\exists j^+(y),j^-(y)\leq N$ s.t. $\lim_{n\rightarrow\infty}d(f^{-n}(y),f^{-n}(p_{j^-(y)}))=0,\lim_{n\rightarrow\infty}d(f^{n}(y),f^{n}(p_{j^+(y)}))=0$. Therefore, $\widehat{\pi}[L]$ is compact, $f$-invariant, and $\chi'$-uniformly hyperbolic for some $\chi'>\chi$ (by the proof of step 1).

\textbf{Step 4:} We now follow the argument in \cite[Lemma~3.12]{BCS}: By the Shadowing Lemma, there are $\epsilon'>0,\delta>0$ s.t. 
\begin{enumerate}
    \item Every $\epsilon'$-pseudo-orbit in $L^\mathbb{Z}$ is $\delta$-shadowed by at least one real orbit \cite[Theorem~18.1.2]{KatokBook}.
    \item Every $\epsilon'$-pseudo-orbit in $L^\mathbb{Z}$ is $2\delta$-shadowed by at most one orbit by expansivity, see \cite[Theorem~18.1.3]{KatokBook} (in particular every orbit as in the first item is unique).
\end{enumerate}
Since $\tilde{L}$ is finite, there is some $m\geq0$ large enough so $d(f^m(y),f^m(p_{j^+(y)}))<\frac{\epsilon'}{2}$,$d(f^-m(y),f^{-m}(p_{j^-(y)}))<\frac{\epsilon'}{2}$, $\forall y\in \tilde{L}$. Let $L_m:=\bigcup_{j=-m}^m\{f^j(y):y\in \widehat{\pi}[\tilde{L}]\}$, which is also finite. Let $$K:=\{x\in M:\text{the orbit of }x\text{ is }\delta\text{-shadowed by an }\epsilon'\text{-pseudo-orbit in }L_m^\mathbb{Z}\}.$$ 
This set contains $\tilde{L}$, and since $L_m$ is finite, it is also closed. It is also invariant and uniformly $\chi'$-hyperbolic for some $\chi'>\chi$ (whence $\subseteq \HWT_\chi^\star$). We construct a point in $K$ with a dense forward-orbit in $K$ in the following way: take a list of all admissible finite words $\{\omega_i\}_{i\geq0}$ with letters in $L_m$.\footnote{That is, $\omega=a_1,...,a_l$ is admissible if $\forall 1\leq i\leq l-1$, $d(f(a_i),a_{i+1})\leq\epsilon'$.} Each $y\in L_m$ connects by such admissible word of length at most $m$, to some periodic point $p_{j^+(y)}\in L_m$, and has some periodic point $p_{j^{-}(y)}$ connecting to it by an admissible word of length at most $m$. For each two periodic points $p_i,p_j\in L_m$,  $p_i$ connects to $p_j$ by an admissible word of length at most $2m$ through $t_{ij}$%, where $p_{i'}=f(p_i)$
. Therefore, every two admissible finite words $\omega,\omega'$ of letters in $L_m$ can be concatenated by some admissible finite word of letters in $L_m$. Concatenate this way all words in $\{\omega_i\}_{i\geq0}$, and take any admissible continuation
to the past. This yields an $\epsilon'$-pseudo-orbit, and the unique orbit in $K$ it $\delta$-shadows must be dense in $K$ by expansivity.\footnote{It can be seen by \cite[Lemma~3.13]{B4}: the mapping $\tau$ which maps each $\epsilon'$-pseudo-orbit to its uniquely $\delta$-shadowed orbit is a continuous map (the topology on the space of pseudo-orbits is the metric topology generated by cylinders); and in addition $\tau\circ\sigma=f\circ\tau$, where $\sigma$ denotes the left-shift on $\epsilon'$-pseudo-orbits. Thus, shadowing longer intervals of the orbit of $x$ forces being closer to it.} Denote this orbit by $o(x)$. 

\textbf{Step 5:} The orbit of $x$ lies in $K$, which is an invariant  $\chi'$-uniformly hyperbolic set for some $\chi'>\chi$ (whence $\|\tilde{C}_\chi^{-1}(\cdot)\|$ is uniformly bounded on $K$), and by the same argument as in footnote \ref{finitelymanyletters}, $x$ has a pre-image in $\Sig^\#$ which involves only finitely many symbols. Choose one pre-image as such, and denote it by $\ul{v}$. There exists at least one symbol $v'$ such that the forward-orbit of $\ul{v}$ visits its cylinder $[v']$ infinitely often. Let $\ul{v}^+:=\sigma^{\min\{i\geq0:\sigma^i\ul{v}\in[v']\}}\ul{v}$, and let $\ul{v}^-$ be some periodic chain in $[v']$. Define 
$
       v_i':=\left\{
  \begin{array}{@{}ll@{}}
    v^+_i, & \text{if}\ i\geq 0\\
    v^-_i, & \text{if}\ i\leq 0
  \end{array}\right.
$, and write $x':=\widehat{\pi}(\ul{v}')$. Then the forward orbit of $x'$ is dense in $K$, and $\ul{v}'\in\tilde{\Sigma}_N:=\{\ul{u}\in \Sig: \ul{u}\in \langle v'\rangle^\mathbb{Z}\}$ (recall Definition \ref{irreducibility}). $\tilde{\Sigma}_N$ is a maximal irreducible component of $\Sig$ containing a compact set which contains the orbit of $\ul{v}'$. For each $y\in \tilde{L}\subseteq K$, the orbit of $x'$ has a subsequence converging to it. Since $\ul{v}'$ is made of finitely-many letters, the orbit of $\ul{v}'$ belongs to a compact subset of $\tilde{\Sigma}_N$. Therefore, the subsequence $\{f^{n_k}(x')\}_{k\geq0}$ which converges to $y$, has a subsubsequence s.t. $\{\sigma^{n_{k_j}}\ul{v}'\}_{j\geq0}$ converges as well. By the continuity of $\widehat{\pi}$, that limit must code $y$. Therefore $\widehat{\pi}[\tilde{\Sigma}_N]\supseteq \tilde{L}$, and moreover, each term in $L$ has a pre-image in $\tilde{\Sigma}_N$ made of finitely-many symbols.

\textbf{Step 6:} For each $N\in \{N_l\}_{l\geq0}$, $p_0\in \widehat{\pi}[\tilde{\Sigma}_N]$. Since $\widehat{\pi}|_{\Sig^\#}$ is finite-to-one, there could be only finitely many maximal irreducible components in $\Sig$, which can code $p_0$. Therefore, there is some subsequence of the maximal irreducible components from step 5, $\tilde{\Sigma}_{N_{l_j}},l_j\uparrow\infty$ which is constant. Denote this component by $\tilde{\Sigma}$. For each fixed $N\in \{N_l\}_{l\geq0}$, we constructed a set $\tilde{L}$, and all such sets must be covered by $\widehat{\pi}[\tilde{\Sigma}]$.

\textbf{Step 7:} Given a point $z\in H(p)\cap \widehat{\pi}[\Sig^{\#\#}]$, consider its coding as in equation \eqref{periodicform}. This coding has a sequence of %trimmed (and then continued periodically) 
finitely-lettered chains which converge to it %, as in the form of the chains constructed 
in $\tilde{L}$% (aside from the chains of the form $\{\ul{\zeta}_{ij}\}_{i,j}$)
. These %trimmed
chains all have their images coded by a chain in $\tilde{\Sigma}$, and they all belong to the same cylinder of their zeroth symbol, denoted by $b$. The corresponding coding chains in $\tilde{\Sigma}$ (which have been shown in step 5 to be made of finitely-many letters, and as such lie in $\Sig^\#$) all belong to a cylinder from $\{[S]\in\mathcal{R}:b\cap\widehat{\pi}[[S]\cap\Sig^\#]\neq\varnothing\}$, which is a finite collection. Thus they have a converging subsequence with a limit in $\tilde{\Sigma}$ (since it is a closed set). By the continuity of $\widehat{\pi}$, that limit must code $z$, and so $\widehat{\pi}[\tilde{\Sigma}]\supseteq H(p)\cap\widehat{\pi}[\Sig^{\#\#}]=H(p)$ modulo all conservative measures which are carried by $H(p)$.
\end{proof}

In \cite[Lemma~3.13]{BCS}, the authors offer a technique of lifting invariant probability measures which are carried by $H(p)$ to an irreducible coding. Their technique involves restriction to points which are generic w.r.t Birkhoff's ergodic theorem, which we may not always be able to do if the measure we wish to lift is infinite. The lifting is being done by the formula $\widehat{\mu}:=\int\frac{1}{|N_x|}\sum_{\ul{R}\in N_x}\delta_{\ul{R}}d\mu(x)$, where $N_x:=\widehat{\pi}^{-1}[\{x\}]\cap \widetilde{\Sigma}^\#$ and $\tilde{\Sigma}^\#:=\{\ul{u}\in\tilde{\Sigma}:\exists v,w\text{ s.t. }\#\{i>0:u_i=v\}=\infty,\text{ and }\#\{i<0:u_i=w\}=\infty\}$. When $\mu$ is carried by $\widehat{\pi}[\widetilde{\Sigma}^\#]$, this lifting is well defined since $\widehat{\pi}$ is finite-to-one on $\Sig^\#$. So, we are required to find a different way which does not depend on generic points to show that all conservative measures which are carried by $H(p)$ are also carried by $\widehat{\pi}[\widetilde{\Sigma}^\#]$. Proposition \ref{irreduciblecoding} was written in a way to guarantee that we can do that. In addition, \cite[Lemma~3.13]{BCS} is done for $HO(p):=\overline{\{q\in H(p):q\text{ is periodic}\}}$ as the homoclinic class which is being lifted to an irreducible component (modulo all invariant ergodic probability measures carried by it). Both Proposition \ref{irreduciblecoding} and Theorem \ref{homoclinicirreducible} are done in the context of ergodic homoclinic classes, which is relevant to specific objects (e.g. SRB measures on ergodic homoclinic classes, see \cite{RodriguezHertz}).

\begin{theorem}\label{homoclinicirreducible}
Let $p$ be a $\chi$-hyperbolic periodic point. Let $\tilde{\Sigma}$ be the irreducible TMS which we construct in Proposition \ref{irreduciblecoding} to cover $H(p)$. Then, $\widehat{\pi}[\tilde{\Sigma}^\#]= H(p)$ modulo all conservative measures which are carried by $H(p)$.
\end{theorem}
\begin{proof} $p$ is clearly in $H(p)\cap\widehat{\pi}[\Sig^{\#\#}]$, and so it is has a coding in $\widetilde{\Sigma}$. $\forall\ul{v}\in\widetilde{\Sigma}^\#$, $\widehat{\pi}(\ul{v})\in H(p)$ by the irreducibility of $\widetilde{\Sigma}$ (and since $\widehat{\pi}[\Sig^\#]=\HWT_\chi^\star$), thus the inclusion $\subseteq$ is evident. We are therefore left to show only the inclusion $\supseteq$. Given each $x\in H(p)\cap \widehat{\pi}[\Sig^{\#\#}](=H(p)$ modulo all conservative measures on $H(p)$), $x$ has a coding $\ul{v}\in \Sig^{\#\#}$ as in equation \eqref{periodicform}. In step 2 of Proposition \ref{irreduciblecoding} we describe a corresponding sequence periodic chains which converge to $\ul{v}$. Call these chains $\{\ul{v}^{(i)}\}_{i\geq0}$. By step 5 of Proposition \ref{irreduciblecoding}, $\{\widehat{\pi}(\ul{v}^{(i)})\}_{i\geq0}$ all have codings made of finitely many letters in $\tilde{\Sigma}$, denoted by $\{\ul{u}^{(i)}\}_{i\geq0}$, and by step 7 of Proposition \ref{irreduciblecoding} we can assume w.l.o.g. that $\{\ul{u}^{(i)}\}_{i\geq0}$ converge to some limit $\ul{u}\in\tilde{\Sigma}$, with $\widehat{\pi}(\ul{u})=x$. Let $\{v_{j_l}\}_{l\in\mathbb{N}}$ (where $j_l\uparrow\infty$,$\{j_l\}_{l\geq0}\subseteq\mathbb{N}$) be a subsequence of the symbols of $\ul{v}$ which is constant (it exists since $\ul{v}\in\Sig^\#$); denote by $w$ the symbol which satisfies $v_{j_l}=w$, $\forall l\geq0$. For all $l\geq0\exists i_l$ s.t. $\forall i>i_l$, $v_{j_l}^{(i)}=v_{j_l}=w$. If $i$ is big enough so $d(\ul{u}^{(i)},\ul{u})\leq e^{-j_l}$, then $u^{(i)}_{j_l}=u_{j_l}$. 
$\ul{v}^{(i)},\ul{u}^{(i)}$ both code the same point and lie in $\Sig^\#$, therefore $u^{(i)}_{j_l}$ is affiliated to $v^{(i)}_{j_l}$. Whence, $w$ is affiliated to $u_{j_l}$. Since $\#\{v':v'\text{ is affiliated to }w\}<\infty$, we get by the pigeonhole principle that some symbol must repeat in $\ul{u}$ for infinitely many positive indices. Similarly it follows that some symbol must repeat in $\ul{u}$ for infinitely many negative indices. Therefore, $\ul{u}\in \tilde{\Sigma}^\#$, and $\widehat{\pi}(\ul{u})=x$.
\end{proof}

\bibliographystyle{alpha}
\bibliography{Elphi}
\end{document}